\documentclass{article}
\usepackage{psfrag}
\usepackage[parfill]{parskip}
\usepackage{float, graphicx}
\usepackage[]{epsfig}
\usepackage{amsmath, amsthm, amssymb}
\usepackage{epsfig}
\usepackage{verbatim}
\usepackage{multicol}
\usepackage{url}
\usepackage{latexsym}
\usepackage{mathrsfs}
\usepackage[colorlinks, bookmarks=true]{hyperref}
\usepackage{graphicx}
\usepackage{amsmath}
\usepackage{enumerate}
\usepackage[normalem]{ulem}
\usepackage{bm}
\usepackage{dsfont}
\usepackage{stmaryrd}
\usepackage{tikz-cd}
\usepackage[T1]{fontenc}

\usepackage{cite}

\usepackage{color}

\makeatletter
\def\thm@space@setup{%
  \thm@preskip=\parskip \thm@postskip=0pt
}
\makeatother

\numberwithin{equation}{section}

\renewcommand{\cal}{\mathcal}

\newcommand\cB{{\mathcal B}}
\newcommand{\cC}{{\cal C}}

\newcommand{\cF}{{\cal F}}
\newcommand{\cG}{{\cal G}}

\newcommand{\cL}{{\cal L}}

\newcommand\cW{{\mathcal W}}

\newcommand{\fa}{{\frak a}}
\newcommand{\fb}{{\frak b}}
\newcommand{\fc}{{\frak c}}

\newcommand{\ft}{{\frak t}}

\newcommand{\fG}{{\frak G}}

\newcommand{\bme}{\bm{e}}

\newcommand{\bmx}{{\bm{x}}}
\newcommand{\bmy}{{\bm{y}}}

\newcommand{\rd}{{\rm d}}

\newcommand{\ri}{\mathrm{i}}

\newcommand{\bC}{{\mathbb C}}

\newcommand{\bE}{\mathbb{E}}

\newcommand{\bP}{\mathbb{P}}

\newcommand{\bR}{{\mathbb R}}

\newcommand{\bZ}{\mathbb{Z}}
\newcommand{\bW}{\mathbb{W}}
\newcommand{\bY}{\mathbb{Y}}

\newcommand{\bM}{\mathbb{M}}
\newcommand{\bH}{\mathbb{H}}

\newcommand{\e}{{\varepsilon}}

\newcommand{\la}{\lambda}

\DeclareMathOperator{\dist}{dist}

\DeclareMathOperator{\OO}{O}
\DeclareMathOperator{\oo}{o}

\renewcommand{\Im}{\mathop{\mathrm{Im}}}

\newcommand{\deq}{\mathrel{\mathop:}=} 
\newcommand{\eqd}{=\mathrel{\mathop:}} 
\renewcommand{\leq}{\leqslant}
\renewcommand{\geq}{\geqslant}

\newcommand{\del}{\partial}

\newcommand{\beq}{\begin{equation}}
\newcommand{\eeq}{\end{equation}}

\theoremstyle{plain} 
\newtheorem{theorem}{Theorem}[section]
\newtheorem*{theorem*}{Theorem}
\newtheorem{lemma}[theorem]{Lemma}
\newtheorem*{lemma*}{Lemma}
\newtheorem{corollary}[theorem]{Corollary}
\newtheorem*{corollary*}{Corollary}
\newtheorem{proposition}[theorem]{Proposition}
\newtheorem*{proposition*}{Proposition}
\newtheorem{assumption}[theorem]{Assumption}
\newtheorem*{assumption*}{Assumption}
\newtheorem{claim}[theorem]{Claim}

\newtheorem*{definition*}{Definition}

\newtheorem*{example*}{Example}
\newtheorem{remark}[theorem]{Remark}

\newtheorem*{remark*}{Remark}
\newtheorem*{remarks*}{Remarks}

\def\author#1{\par
    {\centering{\authorfont#1}\par\vspace*{0.05in}}
}

\def\titlefont{\fontsize{13}{15}\bfseries\boldmath\selectfont\centering{}}
\def\authorfont{\fontsize{13}{15}}

\let\affiliationfont\rhfont

\def\address#1{\par
    {\centering{\affiliationfont#1\par}}\par\vspace*{11pt}
}

\def\body{
\setcounter{footnote}{0}
\def\thefootnote{\alph{footnote}}
\def\@makefnmark{{$^{\rm \@thefnmark}$}}
}

\def\title#1{
    \thispagestyle{plain}
    \vspace*{-14pt}
    \vskip 79pt
    {\centering{\titlefont #1\par}}%
    \vskip 1em
}

\newcommand{\bmla}{\bm{\lambda}}

\newcommand{\cov}{{\rm{cov}}}

\begin{document}
\title{$\beta$-Nonintersecting Poisson Random Walks: Law of Large Numbers and Central Limit Theorems}

\vspace{1.2cm}

 \author{Jiaoyang Huang}
\address{Harvard University\\
   E-mail: jiaoyang@math.harvard.edu}

~\vspace{0.3cm}

\begin{abstract}
We study the $\beta$ analogue of the nonintersecting Poisson random walks. 
We derive a stochastic differential equation of the Stieltjes transform of the empirical measure process, which can be viewed as a dynamical version of the Nekrasov's equation in \cite[Section 4]{MR3668648}.
We find that the empirical measure process 
converges weakly in the space of c{\'a}dl{\'a}g measure-valued processes to a deterministic process, characterized by the quantized free convolution, as introduced in \cite{MR3361772}. 
For suitable initial data, we prove that the rescaled empirical measure process converges weakly in the space of distributions acting on analytic test functions to a Gaussian process. The means and the covariances are universal, and coincide with those of $\beta$-Dyson Brownian motions with the initial data constructed by the Markov-Krein correspondence. Especially, the covariance structure can be described in terms of the Gaussian Free Field. Our proof relies on integrable features of the generators of the $\beta$-nonintersecting Poisson random walks, the method of characteristics, and a coupling technique for Poisson random walks.

\end{abstract}


\begingroup
\hypersetup{linkcolor=black}
 \tableofcontents
\endgroup

\date{\today}

\vspace{-0.7cm}

\section{Introduction}

\subsection{$\beta$-nonintersecting Poisson random walks}
Let $\tilde\bmx(t)=(\tilde x_1(t),\tilde x_2(t),\cdots, \tilde x_n(t))$ be the continuous-time \emph{Poisson random walk} on $\bZ_{\geq 0}^n$, i.e. particles independently jump to the neighboring right site with rate $n$.  The generator of $\bmx(t)$ is given by 
\begin{align*}
\tilde \cL^n f(\tilde \bmx)=\sum_{i=1}^nn\left(f(\tilde\bmx+\bme_i)-f(\tilde \bmx)\right),
\end{align*}
where $\{\bme_i\}_{1\leq i\leq n}$ is the standard vector basis of $\bR^n$. $\tilde \bmx(t)$ conditioned never to collide with each other is the \emph{nonintersecting Poisson random walk}, denoted by $\tilde \bmx(t)=(x_1(t),x_2(t),\cdots,x_n(t))$. The nonintersecting condition has probability zero, and therefore,
needs to be defined through a limit procedure which is performed in \cite{MR1887625}. The nonintersecting Poisson random walk is a continuous time Markov process on 
\begin{align*}
\bW^n_1=\{(\la_1+(n-1),\la_2+(n-2),\cdots,\la_n): (\la_1,\la_2,\cdots,\la_n)\in \bZ_{\geq 0}^n, \la_1\geq\la_2\geq\cdots\geq\la_n\geq 0\},
\end{align*}
with generator
\begin{align*}
\cL_1^n f(\bmx)=n\sum_{i=1}^n\frac{V(\bmx+\bme_i)}{V(\bmx)}\left(f(\bmx+\bme_i)-f(\bmx)\right)=n\sum_{i=1}^{n}\left(\prod_{j:j\neq i}\frac{x_i-x_j+1}{x_i-x_j}\right)\left(f(\bmx+\bme_i)-f(\bmx)\right),
\end{align*}
where $V(\bmx)=\prod_{1\leq i<j\leq n}(x_i-x_j)$
is 
the Vandermond determinant  in variables $x_1,x_2,\cdots,x_n$.

If instead of the Poisson random walk, we start from $n$ independent Brownian motions with mean $0$ and variance $t/n$, then the same conditioning leads to the celebrated \emph{Dyson Brownian motion} with $\beta=2$, which describes the stochastic evolution of eigenvalues of a Hermitian matrix under independent Brownian motion of its entries. For general $\beta>0$, the \emph{$\beta$-Dyson Brownian motion} $\bmy(t)=(y_1(t), y_2(t),\cdots, y_n(t))$ is a diffusion process solving 
\begin{align}\label{e:DBM1}
\rd y_i(t)=\sqrt{\frac{2}{\beta n}}\rd \cB_i(t)+\frac{1}{n}\sum_{j\neq i}\frac{1}{y_i(t)-y_j(t)}\rd t,\quad i=1,2,\cdots, n,
\end{align}
where $\{(\cB_1(t), \cB_2(t),\cdots, \cB_n(t))\}_{t\geq 0}$ are independent standard Brownian motions, and $\{\bmy(t)\}_{t>0}$ lives on the Weyl chamber $\bW^n=\{(\la_1,\la_2,\cdots,\la_n): \la_1>\la_2>\cdots>\la_n\}$. 

The nonintersecting Poisson random walk can be viewed as a discrete version of the Dyson Brownian motion with $\beta=2$. For general $\beta>0$, we
fix $\theta=\beta/2$ and define the \emph{$\beta$-nonintersecting Poisson random walk}, denoted by $\bmx(t)=(x_1(t),x_2(t),\cdots,x_n(t))$, as a continuous time Markov process on 
\begin{align}\label{e:defWtheta}
\bW^n_\theta=\{(\la_1+(n-1)\theta,\la_2+(n-2)\theta,\cdots,\la_n): (\la_1,\la_2,\cdots,\la_n)\in \bZ_{\geq0}^n, \la_1\geq\la_2\geq\cdots\geq\la_n\geq 0\},
\end{align}
with generator
\begin{align}\label{e:generator}
\cL^n_\theta f(\bmx)=\theta n\sum_{i=1}^n\frac{V(\bmx+\theta\bme_i)}{V(\bmx)}\left(f(\bmx+\bme_i)-f(\bmx)\right)=\theta n\sum_{i=1}^{n}\left(\prod_{j:j\neq i}\frac{x_i-x_j+\theta}{x_i-x_j}\right)\left(f(\bmx+\bme_i)-f(\bmx)\right).
\end{align}

In the beautiful article \cite{MR3418747}, Gorin and Shkolnikov constructed certain multilevel discrete Markov chains whose top level dynamics coincide with the $\beta$-nonintersecting Poisson random walks. However, we use slightly different notations, and speed up time by $n$. In \cite{MR3418747}, the $\beta$-nonintersecting Poisson random walks are constructed as stochastic dynamics on Young diagrams. We recall that a Young diagram $\bm\lambda$, is a non-increasing sequence of integers
\begin{align*}
\bm\lambda=(\la_1,\la_2,\la_3, \cdots), \quad \lambda_1\geq \lambda_2\geq\la_3\geq\cdots\geq 0.
\end{align*}
%
We denote $\ell_{\bmla}$ the number of non-empty rows in $\bmla$, i.e. $\la_{\ell_{\bmla}}>0, \la_{\ell_{\bmla}+1}=\la_{\ell_{\bmla}+2}=\cdots =0$, and $|\bmla|=\sum_{i=1}^{\ell_{\bmla}}\la_i$ the number of boxes in $\bmla$.  Let $\bY^n$ denote the set of all Young diagrams with at most $n$ rows, i.e. $\ell_{\bmla}\leq n$.
A box $\Box\in \bmla$ is a pair of integers, 
\begin{align*}
\Box=(i,j)\in {\bm\lambda}, \text{ if and only if } 1\leq i\leq \ell_\lambda, 1\leq j\leq \lambda_i.
\end{align*} 
We denote $\bmla'$ the transposed diagram  of $\bm\lambda$, defined by 
\begin{align*}
\lambda_j'=|\{i: 1\leq j\leq \lambda_i\}|, \quad 1\leq j\leq \la_1.
\end{align*} 
For a box $\Box=(i,j)\in \bmla$, its arm $a_\Box$, leg $l_\Box$, co-arm $a_\Box'$ and co-leg $l_\Box'$ are 
\begin{align*}
a_\Box=\lambda_i-j,\quad l_\Box=\lambda_j'-i,\quad a_\Box'=j-1,\quad l_\Box'=i-1.
\end{align*}

Given a $\beta$-nonintersecting Poisson random walk $\bmx(t)$, we can view it as a growth process on $\bY^n$, by defining $\bmla(t)$ by 
\begin{align}\label{e:defla}
\la_i(t)=x_i(t)-(n-i)\theta, \quad 1\leq i\leq n.
\end{align}
Since $\bmx(t)\in \bW_\theta^n$, we have $\la_1(t)\geq \la_2(t)\geq \cdots\geq \la_n(t)\geq0$, and thus $\bmla(t)$ is a continuous time Markov process on $\bY^n$. Its jump rate is given in \cite[Proposition 2.25]{MR3418747} rescaled by $n$, which, after simplification, is the same as \eqref{e:generator}. There is a simple formula for the transition probability of $\bmla(t)$ with zero initial data \cite[Proposition 2.9, 2.28]{MR3418747}. However, there are no simple formulas for the transition probabilities of $\bmla(t)$ with general initial data.

\begin{theorem}\label{t:density}
Suppose the initial data of $\bmla(t)$ is the empty Young diagram. Then for any fixed $t>0$, the law of $\bmla(t)$ is given by
\begin{align}\label{e:density}
\bP_t(\la_1,\la_2,\cdots, \la_n)=
e^{-\theta t n^2}(\theta tn)^{|\bmla|}\prod_{\Box\in \bmla}\frac{\theta n+a_\Box'-\theta l_\Box'}{(a_\Box+\theta l_\Box+\theta)(a_\Box+\theta l_\Box +1)}.
\end{align}
\end{theorem}

It is proven in {\cite[Theorem 3.2]{MR3418747}} that the Markov process $\bm\la(t)$ converges in the diffusive scaling limit to the $\beta$-Dyson Brownian motion. 
\begin{theorem}
Fix $\theta=\beta/2\geq 1/2$ and let $\varepsilon>0$ be a small parameter. Let $\bmx(t)$ be the $\beta$-nonintersecting Poisson random walk starting at $\bmx(0)\in \bW^n_\theta$. We define $\bm\la(t)$ as in \eqref{e:defla} and the rescaled stochastic process $\bm\la^\varepsilon(t)=(\la_1^\e(t),\la_2^\e(t),\cdots, \la_n^\e(t))$ be defined through,
\begin{align*}
\la_i^\e(t)\deq 
\varepsilon ^{1/2}\left(\frac{\la_i(t/\varepsilon)}{\theta n}-\frac{t}{\varepsilon }\right),\quad i=1,2,\cdots, n.
\end{align*}
  Suppose that as $\e\rightarrow 0$, the initial data $\bm\la^\e(0)$ converges to a point $\bmy(0)\in \bW^n$. Then the process $\bm\la^\e(t)$ converges in the limit $\e\rightarrow 0$ weakly in the Skorokhod topology towards the $\beta$-Dyson Brownian motion $\bmy(t)=(y_1(t), y_2(t),\cdots, y_n(t))$ as in \eqref{e:DBM1}.
\end{theorem}

\subsection{Notations}
Throughout this paper, we use the following notations:

We denote $\bR$ the set of real numbers, $\bC$ the set of complex numbers, $\bH=\bC_+$ the set of complex numbers with positive imaginary parts, $\bC_-$ the set of complex numbers with negative imaginary parts, $\bZ$ the set of integers, and $\bZ_{\geq 0}$ the set of non-negative integers. 

We denote $M_1(\bR)$ the space of probability measures on $\bR$ equipped with the weak topology. A metric compatible  with the weak topology is the \emph{L{\'e}vy metric} defined by
\begin{align*}
{\rm{dist}}(\mu,\nu)\deq\inf_{\varepsilon}\{\mu(-\infty, x-\varepsilon)-\varepsilon\leq \nu(-\infty, x)\leq \mu(-\infty, x+\varepsilon)+\varepsilon \text{ for all $x$}\}.
\end{align*}

Let $(\bM,\dist(\cdot,\cdot))$ be a metric space, either $\bC^m$ or $\bR^m$ with the Euclidean metric or $M_1(\bR)$ with the L{\'e}vy metric. The set of c{\'a}dl{\'a}g functions, i.e. functions which are right continuous with left limits,  from $[0,T]$ to $\bM$ is denoted by $D([0,T],\bM)$ and is called the \emph{Skorokhod space}. Let $\Lambda$ denote the set of all strictly increasing, continuous bijections from $[0,T]$ to $[0,T]$. The \emph{Skorokhod metric} on $D([0,T],\bM)$ is defined by
\begin{align*}
\dist(f,g)=\inf_{\la \in \Lambda}\max\left\{\sup_{0\leq t\leq T}|\la(t)-t|, \sup_{0\leq t\leq T}\dist(f(t), g(\la(t)))\right\}.
\end{align*}
Let $Z^n$, $Z$ be random variables taking value in the Skorokhod space $D([0,T],\bM)$. We say $Z_n$  converges \emph{almost surely} towards $Z$,  if $Z_n\rightarrow Z$ in $D([0,T], \bM)$ for the Skorokhod metric almost surely. We say $Z_n$ \emph{weakly converges} towards $Z$, denoted by $Z^n\Rightarrow Z$, if for all bounded Skorokhod continuous functions $f$,
\begin{align*}
\lim_{n\rightarrow \infty}\bE[f(Z^n)]\rightarrow \bE[f(Z)].
\end{align*}
We refer to \cite[Chapter 1]{MR1431297} and \cite[Chapter 3]{MR1876437} for nice presentations on weak convergence of stochastic processes in the Skorokhod space.

A \emph{random field} is a collection of random variables indexed by elements in a topological space. If the collection of random variables are jointly Gaussian, we call it a \emph{Gaussian random field}. Let $(g^n(z))_{z\in \Omega}$, $(g(z))_{z\in \Omega}$ be $\bC$-valued random fields indexed by an open subset $\Omega\subset \bC\setminus\bR$. We say $(g^n(z))_{z\in \Omega}$ weakly converges towards $(g(z))_{z\in \Omega}$ in the sense of finite dimensional distributions, if for any $z_1,z_2,\cdots, z_m\in \Omega$ the random vector $(g^n(z_j))_{1\leq j\leq m}$ weakly converges to $(g(z_j))_{1\leq j\leq m}$. Let $\{(g_t^n(z))_{z\in \Omega}\}_{0\leq t\leq T}$, $\{(g_t(z))_{z\in \Omega}\}_{0\leq t\leq T}$ be random field valued random processes. We say $\{(g_t^n(z))_{z\in \Omega}\}_{0\leq t\leq T}$ weakly converges towards $\{(g_t(z))_{z\in \Omega}\}_{0\leq t\leq T}$ in the sense of finite dimensional processes, if for any $z_1,z_2,\cdots, z_m\in \Omega$ the random process $\{(g_t^n(z_j))_{1\leq j\leq m}\}_{0\leq t\leq T}$ weakly converges to $\{(g_t(z_j))_{1\leq j\leq m}\}_{0\leq t\leq T}$ in the Skorokhod space $D([0,T], \bC^m)$. 

\subsection{Main results}

In this paper, we study the asymptotic behavior of the $\beta$-nonintersecting Poisson random walks, as the number of particles $n$ goes to infinity.

We consider $\beta$-nonintersecting Poisson random walks $\bmx(t)=(x_1(t), x_2(t),\cdots, x_n(t))$, with initial data $\bmx(0)=(x_1(0), x_2(0),\cdots, x_n(0))$. We define the empirical measure process 
\begin{align}\label{e:defemp}
\mu^n_t=\frac{1}{n}\sum_{i=1}^n \delta_{x_i(t)/\theta n},
\end{align}
which can be viewed as a random element in $D([0,T], M_1(\bR))$, the space of  right-continuous with left limits processes from $[0,T]$ into the space $M_1(\bR)$ of probability measures on $\bR$.

The law of large numbers theorem states that the empirical measure process $\{\mu_t^n\}_{0\leq t\leq T}$ converges in $D([0,T], M_1(\bR))$. We need to assume that the initial empirical measure $\mu_0^n$ converges in the L{\'e}vy metric as $n$ goes to infinity in $M_1(\bR)$.

We denote the Stieltjes transform of the empirical measure at time $t$ as
\begin{align}\label{e:defmt}
m^n_t(z)=\frac{1}{n}\sum_{i=1}^{n}\frac{1}{x_i(t)/\theta n-z}=\int \frac{\rd \mu^n_t(x)}{x-z},
\end{align}
where $z\in \bC\setminus \bR$. 

\begin{theorem}\label{t:LLN}
Fix $\theta>0$. We assume that there exists a measure $\mu_0\in M_1(\bR)$, such that the initial empirical measure $\mu_0^n$ converges in the L{\'e}vy metric as $n$ goes to infinity towards $\mu_0$ almost surely (in probability).
Then, for any fixed time $T>0$, $\{\mu^n_t\}_{0\leq t\leq T}$ converges as $n$ goes to infinity  in $D([0,T], M_1(\bR))$ almost surely (in probability). Its limit is the unique measure-valued process $\{\mu_t\}_{0\leq t\leq T}$, so that the density satisfies $0\leq \rd \mu_t(x)/\rd x\leq 1$, and its Stieltjes transform 
\begin{align}\label{e:defmt}
 m_t(z)=\int \frac{\rd \mu_t(x)}{x-z},
\end{align} 
satisfies the equation
\begin{align}\label{e:limitMST}
m_t(z)=m_0(z)-\int_0^t e^{- m_s(z)}\del_z m_s(z)\rd s,
\end{align}
for $z\in \bC\setminus \bR$.
\end{theorem}
\begin{remark}
Assumption in Theorem \ref{t:LLN} is equivalent to that the Stieltjes transform of the initial empirical measure 
\begin{align*}
\lim_{n\rightarrow \infty} m^n_0(z)=\int \frac{\rd\mu_0(x)}{x-z}\eqd m_0(z),
\end{align*}
almost surely (in probability), for any $z\in \bC\setminus \bR$.
\end{remark}

The Stieltjes transform of  $ \mu_t$ is characterized by \eqref{e:limitMST},
\begin{align}\label{e:dif1}
\del_t m_t(z)=-e^{-m_t(z)}\del_z  m_t(z)=\del_z (e^{- m_t(z)}).
\end{align} 
This is a complex Burgers type equation, and can be solved by the method of characteristics. We define the characteristic lines,
\begin{align}\label{e:dif2}
\del_t  z_t(z)=e^{- m_t( z_t(z))}, \quad z_0(z)=z.
\end{align}
If the context is clear, we omit the parameter $z$, i.e. we simply write $z_t$ instead of $z_t(z)$. Plugging \eqref{e:dif2} into \eqref{e:dif1}, and applying the chain rule we obtain 
$
\del_t  m_t(z_t)=0.
$
It implies that $ m_t(z)$ is a constant along the characteristic lines, i.e. $m_t(z_t(z))=m_0(z_0(z))=m_0(z)$. And the solution of the differential equation \eqref{e:dif2} is given by
\begin{align} \label{e:deft}
z_t(z)=z+te^{-m_0(z)}, \quad 0\leq t< -\frac{\Im[z]}{\Im[e^{-m_0(z)}]}\eqd \ft(z).
\end{align}
We conclude that the Stieltjes transform $m_t(z)$ is given by
\begin{align}\label{e:tmtrelation}
m_t(z+te^{-m_0(z)})=m_0(z).
\end{align}
Later we will prove that for any time $t\geq 0$, there exists an open set $\Omega_t\subset\bC\setminus \bR$ defined in \eqref{e:defOmega}, such that   $z_t(z)=z+te^{-m_0(z)}$ is conformal from $\Omega_t$ to $\bC\setminus \bR$, and is a homeomorphism from the closure of $\Omega_t\cap \bC_+$ to $\bC_+\cup \bR$, and from the closure of $\Omega_t\cap \bC_-$ to $\bC_-\cup \bR$. 
 
The central limit theorem states that the rescaled empirical measure process $\{n(\mu_t^n-\mu_t)\}_{0\leq t\leq T}$ weakly converges in the space of distributions acting on analytic test functions to a Gaussian process. We need to assume that the rescaled initial empirical measure $n(\mu_0^n-\mu_0)$ weakly converges to a measure.

We define the rescaled fluctuation process
\begin{align}\label{e:defgt}
g_t^n(z)=n(m_t^n(z)- m_t(z))=n\int \frac{\rd (\mu_t^n(x)- \mu_t(x))}{x-z},
\end{align}
which characterizes the behaviors of the rescaled empirical measure process $\{n(\mu_t^n-\mu_t)\}_{0\leq t\leq T}$.

\begin{assumption}\label{a:initiallaw}
We assume there exists a constant $\fa$, such that that for any $z\in \bC\setminus\bR$, 
\begin{align*}
\bE\left[|g_0^n(z)|^2\right]\leq \fa(\Im[z])^{-2},
\end{align*}
and the random field $(g_0^n(z))_{z\in \bC\setminus\bR}$ weakly converges to a deterministic field $(g_0(z))_{z\in \bC\setminus \bR}$, in the sense of finite dimensional distributions.
\end{assumption}

\begin{remark}
Assumption \ref{a:initiallaw} implies that the initial empirical measure $\mu_0^n$ converges in the L{\'e}vy metric as $n$ goes to infinity towards $\mu_0$ in probability.
\end{remark}

\begin{theorem}\label{t:CLT}
Fix $\theta>0$. We assume Assumption \ref{a:initiallaw}. Then for any fixed time $T>0$,  the process $\{(g^n_t(z_t(z)))_{z\in\Omega_T}\}_{0\leq t\leq T}$ converges weakly towards a Gaussian process $\{(g_t(z_t(z)))_{z\in \Omega_T}\}_{0\leq t\leq T}$, in the sense of finite dimensional processes, with initial data $(g_0(z))_{z\in \bC\setminus\bR}$ given in Assumption \ref{a:initiallaw}, means 
\begin{align}\label{e:mean0}
\bE[g_t(z_t(z))]&=\mu(t, z)\deq\frac{g_0(z)}{1-t\del_zm_0(z)e^{-m_0(z)}}
+\left(\frac{1}{2}-\frac{1}{2\theta}\right)\frac{t((\del_z m_0(z))^2-\del_z^2 m_0(z))e^{-m_0(z)}}{(1-t\del_z m_0(z)e^{-m_0(z)})^2},
\end{align}
and covariances
\begin{align}
\begin{split}\label{e:variance0}
\cov[g_s(z_s(z)), g_t(z_t(z'))]&=\sigma(s, z, t, z')\deq\frac{1}{\theta}\frac{1}{(1-s\del_zm_0(z)e^{-m_0(z)})(1-t\del_zm_0(z')e^{-m_0(z')})}\\
&\times\left(\frac{1}{(z-z')^2}-\frac{(1-(s\wedge t)\del_zm_0(z)e^{-m_0(z)})(1-(s\wedge t)\del_zm_0(z')e^{-m_0(z')})}{(z-z'+(s\wedge t)(e^{-m_0(z)}-e^{-m_0(z')}))^2}\right)\\
\cov[g_s(z_s(z)), \overline{g_t(z_t(z'))}]&=\sigma(s, z,t,\bar z'),
\end{split}
\end{align}
where 
\begin{align*}
\sigma(s, z,t,z)\deq \lim_{z'\rightarrow z}\sigma(s,z,t,z')
&=\frac{(s\wedge t)e^{-m_0(z)}(2(\del_z m_0(z))^3-6\del_z m_0(z)\del_z^2 m_0(z)+2\del_z^3m_0(z))}{12\theta(1-(s\wedge t)\del_z m_0(z)e^{-m_0(z)})^3(1-(s\vee t)\del_z m_0(z)e^{-m_0(z)})}\\
&+\frac{(s\wedge t)^2e^{-2m_0(z)}((\del_z m_0(z))^4+3(\del_z^2 m_0(z))^2-2\del_z m_0(z)\del_z^3m_0(z))}{12\theta(1-(s\wedge t)\del_z m_0(z)e^{-m_0(z)})^3(1-(s\vee t)\del_z m_0(z)e^{-m_0(z)})}.
\end{align*}
\end{theorem}


\begin{remark}
We can rewrite the means and covariances in terms of the characteristic lines $z_t(z)$:
\begin{align}\begin{split}\label{e:meanandvar}
\mu(t,z)&=\frac{g_0(z)}{\del_z z_t(z)}+\left(\frac{1}{2}-\frac{1}{2\theta}\right)\frac{\del_z^2 z_t(z)}{(\del_z z_t(z))^2},\\
\sigma(s, z, t,z')&=\frac{1}{\theta}\frac{1}{\del_zz_s(z)\del_zz_t(z')}\left(\frac{1}{(z-z')^2}-\frac{\del_z z_{s\wedge t}(z)\del_z z_{s\wedge t}(z')}{(z_{s\wedge t}(z)-z_{s\wedge t}(z'))^2}\right)\\
&=\frac{1}{\theta}\frac{1}{\del_zz_s(z)\del_zz_t(z')}\del_z\del_{z'}\log \left(\frac{z-z'}{z_{s\wedge t}(z)-z_{s\wedge t}(z')}\right).
\end{split}\end{align}
We will prove in Section \ref{s:compareDBM}, the means and the covariances \eqref{e:meanandvar} are universal, and coincide with those of $\beta$-Dyson Brownian motions with initial data constructed by the Markov-Krein correspondence.
\end{remark}

To study the fluctuation of the rescaled empirical measure process $\{n(\mu_t^n-\mu_t)\}_{0\leq t\leq T}$ with analytic functions as test functions, we need to assume that the extreme particles are bounded. 

\begin{assumption}\label{a:ibound}
We assume there exists a large number $\fb$, such that 
\begin{align}\label{e:ibound}
\fb n\geq x_1(0)\geq x_2(0)\geq \cdots \geq x_n(0).
\end{align}
\end{assumption}

\begin{theorem}\label{t:CLT2}
Fix $\theta>0$. We assume Assumptions \ref{a:initiallaw} and \ref{a:ibound}. Then for any fixed time $T>0$ and real analytic functions $f_1, f_2, \cdots, f_m$ on $\bR$, the random process 
\begin{align*}
\left\{\left(n\int f_j(x)\rd(\mu_{t}^{n}(x)-\mu_{t}(x))\right)_{1\leq j\leq m}\right\}_{0\leq t\leq T},
\end{align*} 
converges as $n$ goes to infinity in $D([0,T], \bR^m)$ weakly towards a Gaussian process $\{(\cF_j(t))_{1\leq j\leq m}\}_{0\leq t\leq T}$, with means and covariances
\begin{align*}
\bE[\cF_j(t)]&=\frac{1}{2\pi\ri}\oint_\cC \mu(t, z_t^{-1}(w))f_j(w)\rd w,\\
\cov[\cF_j(s), \cF_k(t)]&=-\frac{1}{4\pi^2}\oint_{\cC} \oint_{\cC} \sigma(s, z_s^{-1}(w),t, z_t^{-1}(w'))f_j(w)f_k(w')\rd w\rd w',
\end{align*}
where the contours are sufficiently large depending on $\fb$.
\end{theorem}

\begin{remark}
We prove in Proposition \ref{p:extremePbound} that with exponentially high probability all particles $x_i(t)$ are inside an interval $[0, \fc n]$. The contours in Theorem \ref{t:CLT2} encloses a neighborhood of $[0,\fc/\theta]$. Further, it is enough to assume in Theorem \ref{t:CLT2} that $f_j$ are analytic only in a neighborhood of $[0,\fc/\theta]$.
\end{remark}

As a corollary of Theorem \ref{t:CLT2}, we can identify the fluctuation of the rescaled empirical measure process $\{n(\mu_t^n-\mu_t)\}_{0\leq t\leq T}$ with the Gaussian Free Field. Let $\bH=\{z\in \bC: \Im[z]>0\}$ be the upper half plane. The Gaussian Free Field $\fG$ on $\bH$ with zero boundary conditions, see e.g. \cite{MR2322706}, is a probability measure on a suitable class of generalized functions on $\bH$, and  can be characterized as follows. We take any sequence $\{\phi_k\}_{k\geq 1}$ of compactly supported test functions, the pairings
\begin{align*}
\int_\bH \phi_k(z)\fG(z)|\rd z|^2\deq \fG(\phi_k),\quad k\geq 1,
\end{align*}
form a sequence of mean $0$ Gaussian random variables with covariance matrix
\begin{align*}
\bE\left[\fG(\phi_k), \fG(\phi_l)\right]
=\int_{\bH}|\rd z_1|^2 |\rd z_2|^2 \phi_k(z_1)\phi_l(z_2) G(z_1,z_2),
\end{align*}
where 
\begin{align*}
G(z,w)=-\frac{1}{2\pi}\ln \left|\frac{z-w}{z-\bar w}\right|.
\end{align*}
is the Green function of the Laplacian on $\bH$ with Dirichlet boundary conditions.  One can make sense of the integrals
$\int
f(z)\fG(z)\rd z$ over finite contours in $\bH$ with continuous functions $f(z)$.

Given a $\beta$-nonintersecting Poisson random walk $\bmx(t)=(x_1(t),x_2(t),\cdots, x_n(t))$, with initial data $\bmx(0)=(x_1(0), x_2(0),\cdots, x_n(0))$, we introduce the height function $H_n:  \bR\times \bR_{\geq 0}\mapsto \bZ_{\geq 0}$:
\begin{align*}
H_n(x,t)=|\{1\leq i\leq n: x_i(t)\geq \theta n x\}|.
\end{align*}
We define the map $z\mapsto x(z)+\ri t(z)$ from $\bH$ to $\bH$,
\begin{align*}
 (x(z),t(z))=\left(\frac{ze^{-m_0(\bar z)}-\bar z e^{-m_0(z)}}{e^{-m_0(\bar z)}-e^{-m_0(z)}}, \frac{z-\bar z}{e^{-m_0(\bar z)}-e^{-m_0(z)}}\right).
\end{align*}
We note that the expressions for $x(z)$ and $t(z)$ are invariant with respect to complex conjugate, so $x(z)$ and $t(z)$ are indeed real for any $z\in \bH$. Since $x(z)=z+t(z)e^{-m_0(z)}\in \bR$, the map $z\mapsto (x(z),t(z))$ is in fact a diffeomorphism from $\bH$ to its image. We define the pull back height function on $\bH$,
\begin{align}\label{e:defHn}
H_n(z)\deq H_n(x(z), t(z)),\quad z\in \bH.
\end{align}
One might worry that some information is lost in this transformation, as the image of
the map $z\mapsto (x(z),t(z))$ is smaller than $\bH$, yet the height function $H_n(x,t)$ is actually
frozen outside this image and there are no fluctuations to study. Next corollary states that the height function $H_n(z)$ converges to the Gaussian Free Field on $\bH$ with zero boundary conditions.

\begin{corollary}\label{c:GFF}
Fix $\theta>0$. We assume Assumptions \ref{a:initiallaw} and \ref{a:ibound}. Let $H_n(z)$ be the random height function on $\bH$ as defined in \eqref{e:defHn}, then as $n$ goes to infinity,
\begin{align*}
\sqrt{\pi\theta}\left(H_n(z)-\bE[H_n(z)]\right)\rightarrow \fG(z).
\end{align*}
In more details, for any fixed time $T>0$, and real analytic functions $f_1,f_2,\cdots, f_m$ on $\bR$, the random process,
\begin{align*}
\left\{\left(\sqrt{\pi\theta}\int_\bR f_j(x) (H_n(x, t)-\bE[H_n(x,t)])\rd x\right)_{1\leq j\leq m} \right\}_{0\leq t\leq T},
\end{align*}
converges as $n$ goes to infinity in $D([0,T], \bR^m)$ weakly towards a Gaussian process 
\begin{align*}
\left\{
\left(\int_{z\in \bH, t(z)=t} f_j(x(z))\fG(z) \rd x(z)\right)_{1\leq j\leq m}
\right\}_{0\leq t\leq T}.
\end{align*}
\end{corollary}

\subsection{Related results}

For the $\beta$-Dyson Brownian motion \eqref{e:DBM1}, the asymptotic behavior of the empirical measure process was studied in 
\cite{MR1217451,MR1176727}. They found that the empirical measure process 
\begin{align}\label{e:empirical}
\tilde \mu_t^n=\frac{1}{n}\sum_{i=1}^n \delta_{y_i(t)},\quad 1\leq i\leq n,
\end{align}
converges weakly in the space of continuous measure-valued processes to a deterministic process $\tilde \mu_t$, characterized by the free convolution with semi-circle distributions. It was proven in \cite{MR1819483}, that the rescaled empirical measure process $n(\tilde \mu_t^n-\tilde \mu_t)$ converges weakly in the space of distributions acting on a class of $C^6$ test functions to a Gaussian process, provided that the initial distributions $n(\tilde\mu_0^n-\tilde\mu_0)$ converge. The explicit formulas of the means and the covariances of the limit Gaussian process was derived in \cite{MR2418256}.


More generally, the $\beta$-Dyson Brownian motion $\bmy(t)=(y_1(t), y_2(t),\cdots, y_n(t))$ with potential $V$ is given by
\begin{align*}
\rd y_i(t)=\sqrt{\frac{2}{\beta n}}\rd \cB_i(t)+\frac{1}{n}\sum_{j\neq i}\frac{1}{y_i(t)-y_j(t)}\rd t -\frac{1}{2}V'(y_i(t))\rd t,\quad i=1,2,\cdots, n,
\end{align*}
where $\{(\cB_1(t), \cB_2(t), \cdots, \cB_n(t))\}_{t\geq 0}$ are independent standard Brownian motions. It was proven in \cite{GDBM1, GDBM2}, that under mild conditions on $V$, the empirical measure process converges to a $V$-dependent measure-valued process,  which can be realized as the gradient flow of the Voiculescu free entropy on the Wasserstein space over $\bR$. The central limit theorem of the rescaled empirical measure process was proven in \cite{Un} for $\beta>1$ and sufficiently regular convex potential $V$.

The Wigner-Dyson-Mehta conjecture stated that the eigenvalue correlation functions of a general
class of random matrices converge to the corresponding
ones of Gaussian matrices. The Dyson Brownian motion  \eqref{e:DBM1} plays a central role in the three-step approach to the universality conjecture in a series of works \cite{Landon2016,fix,MR2919197,MR3429490,MR3372074,MR2810797,kevin3}, developed by Erd{\H o}s, Yau and their collaborators. Parallel results were established in certain cases
in \cite{MR2669449,MR2784665}, with a four moment comparison theorem.

The transition probability of the $\beta$-nonintersecting Poisson random walks with the fully-packed initial data $x_i(0)=(n-i)\theta$, $1\leq i\leq n$, is a discrete $\beta$ ensemble with Charlier weight. The discrete $\beta$ ensembles with general weights were introduced in \cite{MR3668648}, which is a probability distribution
\begin{align}\label{e:disc}
\bP_n(\ell_1,\ell_2,\cdots, \ell_n)=\frac{1}{Z_n}\prod_{1\leq i<j\leq n}\frac{\Gamma(\ell_i-\ell_j+1)\Gamma(\ell_i-\ell_j+\theta)}{\Gamma(\ell_i-\ell_j)\Gamma(\ell_i-\ell_j+1-\theta)}\prod_{i=1}^nw(\ell_i; N),
\end{align}
on ordered $n$-tuples $\ell_1>\ell_2>\cdots \ell_n$ such that $\ell_i=\la_i+(n-i)\theta$ and $\la_1\geq \la_2\geq\cdots\geq \la_N$ are integers. The discrete $\beta$ ensembles are discretizations for the $\beta$ ensembles of random matrix theory, which are probability distributions on $n$ tuples of reals $y_1>y_2>\cdots>y_n$,
\begin{align}\label{e:betaensemble}
\bP_n(y_1,y_2,\cdots, y_n)=\frac{1}{Z_n}\prod_{1\leq i<j\leq n}|y_i-y_j|^\beta\prod_{i=1}^ne^{-nV(y_i)},
\end{align}
where the potential $V$ is a continuous function. Under mild assumptions on the potential $V$, the $\beta$ ensembles \eqref{e:betaensemble} exhibit a law of large number, i.e., the empirical measure 
\begin{align*}
\mu^n=\frac{1}{n}\sum_{i=1}^n \delta_{y_i},
\end{align*}
converges to a non-random equilibrium measure $\mu$. For $\beta=1,2,4$ and $V(y)=y^2$, this statement dates back to the original work of Wigner \cite{MR0077805,MR0083848}. We refer to \cite[Chapter 2.6]{MR2760897} for the study of the $\beta$ ensembles with general $V$. In the breakthrough paper \cite{MR1487983}, Johansson introduced  the loop (or Dyson-Schwinger) equations to the mathematical community, and proved that the rescaled empirical measure satisfies a central limit theorem, i.e., for sufficiently smooth functions $f(y)$ the random variable
\begin{align*}
n\int f(x) (\rd \mu^n(x)-\rd\mu(x)).
\end{align*}
converges to a Gaussian random variable. We refer to \cite{MR3010191, borot-guionnet2, KrSh} for further development. The law of large numbers and the central limit theorems of the discrete $\beta$ ensemble \eqref{e:disc} were proven in \cite{MR3668648}, using a discrete version of the loop equations \cite{Nekrasov}.

In the special case when $\beta=2$, the central limit theorem for the global fluctuations of the nonintersecting Poisson random walk were obtained by various methods. For the fully-packed initial data, the central limit theorem was established in \cite{MR3148098} by the technique of determinantal point processes, in \cite{MR3552537,MR3263029} by computations in the universal enveloping algebra of
$U(N)$ and in \cite{Dui, MR3556288} by employing finite term recurrence relations of orthogonal polynomials. For general initial data, the law of large numbers and the central limit theorems were proven in \cite{MR3361772,BuGo}, where the Schur
generating functions were introduced to study random discrete models. Our results give a new proof of these results based on the dynamical approach. 

\subsection{Organization of the paper}

In Section \ref{s:qfc}, we recall the quantized free convolution as introduced in \cite{MR3361772}. We show that the limit measure-valued process $\mu_t$ is characterized by the quantized free convolution. 
In Section \ref{s:compareDBM}, we compare the central limit theorems of the $\beta$-nonintersecting Poisson random walks with those of the $\beta$-Dyson Brownian motions. It turns out that the means and covariances of the limit fluctuation process coincide under Markov-Krein correspondence. 
In Section \ref{s:sde}, we collect some properties of the generator $\cL_\theta^n$ of the $\beta$-nonintersecting Poisson random walks, and derive a stochastic differential equation of the Stieltjes transform of the empirical measure process, which relies on the integrable features of the generator. The stochastic differential equation can be viewed as a dynamical version of the Nekrasov's equation in \cite[Section 4]{MR3668648}, which is crucial for the proof of central limit theorems of the discrete $\beta$ ensembles.
 In Section \ref{s:LLN} and \ref{s:CLT}, we prove the law of large numbers and central limit theorem of the $\beta$-nonintersecting Poisson random walks. We directly analyze the stochastic differential equation satisfied by the Stieltjes transform using the method of characteristics as in \cite{HL}, where the method of characteristics was used to derive the rigidity of the Dyson Brownian motion. Since the $\beta$-nonintersecting Poisson random walks are jump processes, the analysis is more sophisticated than that of the Dyson Brownian motion. In Section \ref{s:extremep} we derive an estimate of the locations of extreme particles, by a coupling technique, and prove the central limit theorem with analytic test functions. 

Finally we remark that by analyzing the stochastic differential equation of the Stieltjes transform of the empirical measure process as in \cite{HL}, one
can prove the optimal rigidity estimates and a mesoscopic central limit theorem for the $\beta$-nonintersecting Poisson random walks.

\noindent\textbf{Acknowledgement.} The author heartily thanks Vadim Gorin for constructive comments on the draft of this
paper.

\section{Law of large numbers for the empirical measure process}

\subsection{Quantized free convolution}\label{s:qfc}
In this section we study the limit measure-valued process $ \mu_t$. To describe it, we need the concept of quantized free convolution as introduced in \cite{MR3361772}. The quantized free convolution is a quantized version of the free convolution originally defined by Voiculescu \cite{MR799593,MR839105} in the setting of operator algebras. Given a probability measure $\mu$, we denote its Stieltjes transform by $m_\mu(z)=\int \rd \mu(x)/(x-z)$, for any $z\in \bC\setminus\bR$. The $R$-transform is defined as
\begin{align}\label{e:defRf}
R_\mu(z)\deq m_\mu^{-1}(-z)-\frac{1}{z},
\end{align}
where $m_\mu^{-1}(z)$ is the functional inverse of $m_\mu(z)$, i.e. $m_\mu(m_{\mu}^{-1}(z))=m_{\mu}^{-1}(m_\mu(z))=z$. 
The free convolution is a unique operation on probability measures $(\mu, \nu)\mapsto \mu\boxplus\nu$, which agrees with the addition of the $R$-transforms:
\begin{align*}
R_\mu(z)+R_{\nu}(z)=R_{\mu\boxplus\nu}(z)
\end{align*}
It was proven in \cite{MR1094052} that the asymptotic distribution of eigenvalues of sums of independent random matrices is given by the free convolution.

The quantized free convolution is an operation on probability measures which have bounded by $1$ density with respect to the Lebesgue measure. One gets the quantized
free convolution by replacing the $R$-transform in \eqref{e:defRf} with the quantized $R$-transform
\begin{align}\label{e:defR}
R_\mu^{\text{quant}}(z)\deq m_\mu^{-1}(-z)-\frac{1}{1-e^{-z}}.
\end{align}
The quantized free convolution is a unique operation on probability measures $(\mu, \nu)\mapsto \mu\otimes\nu$, which agrees with the addition of the quantized $R$-transforms:
\begin{align*}
R_\mu^{\text{quant}}(z)+R_{\nu}^{\text{quant}}(z)=R_{\mu\otimes\nu}^{\text{quant}}(z)
\end{align*}
It was proven in \cite[Theorem 1.1]{MR3361772} that the quantized free convolution characterizes the tensor product of two irreducible representation of unitary group.
 
The Markov-Krein correspondence \cite{MR1618739, MR0167806} gives an exact relationship
between the free convolution and the quantized free convolution. 
\begin{theorem}\label{t:MKcor}
For every probability measure $\mu$ on $\bR$ which has bounded by $1$ density with respect to the Lebesgue measure,  there exists a probability measure $Q(\mu)$ such that 
\begin{align}\label{e:defQ}
m_{Q(\mu)}(z)=1-e^{-m_{\mu}(z)},
\end{align}
where $m_\mu(z)$ and $m_{Q(\mu)}(z)$ are Stieltjes transforms of $\mu$ and $Q(\mu)$ respectively. We denote the operator 
\begin{align*}
\tilde Q(\mu)=r\circ Q \circ r(\mu),
\end{align*}
where $r$ is the reflection of a measure with respect to the origin. The operator $\tilde Q$ intertwines the free convolution and the quantized free convolution, i.e. for any two probability measures $\mu_1, \mu_2$ as above, we have
\begin{align*}
\tilde Q(\mu_1\otimes\mu_2)=\tilde Q(\mu_1)\boxplus \tilde Q(\mu_2).
\end{align*}
\end{theorem}
Theorem \ref{t:MKcor} essentially reduces the quantized free convolution to the free convolution. Properties of the quantized free convolution, e.g., existence and uniqueness, follow from their counterparts of the free convolution. We sketch the construction of the operator $Q$ in Remark \ref{r:proofMKcor} in Section \ref{s:sde}.

%
%
%

The limit measure-valued process $ \mu_t$ can be described by the  quantized free convolution. We denote $ R_t^{\text{quant}}(z)$ the quantized $R$-transform of the measure $ \mu_t$. From \eqref{e:tmtrelation}, we have
\begin{align*}
\left(m_t\right)^{-1}(z)=\left(m_0\right)^{-1}(z)+te^{-z},
\end{align*}
and 
\begin{align*}
 R_t^{\text{quant}}(z)= R_0^{\text{quant}}(z)+te^{z}.
\end{align*}
There exists a family of measures $ \nu_t$ such that the quantized $R$-transform of $ \nu_t$ is given by $te^{z}$. The Stieltjes transform $m_{\nu_t}(z)$ of $\nu_t$ is given by
\begin{align*}
z e^{2m_{\nu_t}(z)}+(1-t-z)e^{m_{\nu_t}(z)}+t=0.
\end{align*}
We can solve for $m_{\nu_t}(z)$, and the density of $ \nu_t$ is given by for $t\leq 1$,
\begin{align}\label{e:defnu1}
\rd  \nu_t(x)/\rd x=
\left\{
\begin{array}{cc}
\frac{1}{\pi} \text{arccot} \left(\frac{x+t-1}{\sqrt{4xt-(x+t-1)^2}}\right), &(1-\sqrt{t})^2\leq x\leq (\sqrt{t}+1)^2, \\
1,& x< (1-\sqrt{t})^2,\\
0, & x> (\sqrt{t}+1)^2,
\end{array}
\right.
\end{align}
for $t>1$,
\begin{align}\label{e:defnu2}
\rd  \nu_t(x)/\rd x=
\left\{
\begin{array}{cc}
\frac{1}{\pi} \text{arccot} \left(\frac{x+t-1}{\sqrt{4xt-(x+t-1)^2}}\right), &(1-\sqrt{t})^2\leq x\leq (\sqrt{t}+1)^2, \\
0, & x< (\sqrt{t}-1)^2 \text{ or } x> (\sqrt{t}+1)^2.
\end{array}
\right.
\end{align}
We can conclude from the discussion above, 
\begin{proposition}
The limit measure $ \mu_t$ is the quantized free convolution of the initial measure $ \mu_0$ with the measure $ \nu_t$ as defined in \eqref{e:defnu1} and \eqref{e:defnu2}:
\begin{align*}
 \mu_t= \mu_0\otimes \nu_t.
\end{align*}
\end{proposition}

In the rest of this section, we collect some properties of the Stieltjes transform $m_t$, the characteristic lines $z_t$ and the logarithmic potential $h_t$ of the measure $\mu_t$,
\begin{align}\label{e:defht}
h_t(z)=\int\log(x-z)\rd \mu_t(x), \quad z\in \bC\setminus \bR.
\end{align}
We remark that $\del_z h_t(z)=-m_t(z)$.

\begin{proposition}\label{p:ztp}
For any time $t\geq 0$, we define an open set $\Omega_t\subset\bC\setminus \bR$ 
\begin{align} \label{e:defOmega}
\Omega_t\deq\left\{z\in \bC\setminus \bR: \int \frac{\rd Q(\mu_0)(x)}{|x-z|^2}<\frac{1}{t}\right\},
\end{align}
where the operator $Q$ is defined in Theorem \ref{t:MKcor}. Then,  $z_t(z)=z+te^{-m_0(z)}$ is conformal from $\Omega_t$ to $\bC\setminus \bR$, and is a homeomorphism from the closure of $\Omega_t\cap \bC_+$ to $\bC_+\cup \bR$, and from the closure of $\Omega_t\cap \bC_-$ to $\bC_-\cup \bR$. 
Moreover for any $z\in \Omega_t$, $|\Im[z_s]|$ is monotonically decreasing for $0\leq s\leq t$, i.e., $|\Im[z_s]|\geq |\Im[z_t]|$.
\end{proposition}
\begin{proof}
Thanks to Theorem \ref{t:MKcor}, we have
\begin{align*}
z_t(z)=z+te^{-m_0(z)}=z+t-tm_{Q(\mu_0)}(z),
\end{align*}
and the proposition follows from \cite[Lemma 4]{MR1488333}.
\end{proof}

\begin{proposition}\label{p:estimatemt}
Fix $T>0$. For any $0\leq t\leq T$ and $z\in \Omega_T$ as defined in \eqref{e:defOmega}, we have
\begin{align}\begin{split}\label{e:mtestimate}
&(\del_z m_t)(z_t(z))=\frac{\del_zm_0(z)}{1-t\del_z m_0(z)e^{-m_0(z)}},\quad 
(\del^2_z m_t)(z_t(z))=\frac{\del_z^2 m_0(z)-t(\del_zm_0(z))^3e^{-m_0(z)}}{(1-t\del_z m_0(z)e^{-m_0(z)})^3},\\
&(\del_t m_t)(z_t(z))=-\frac{\del_z m_0(z)e^{-m_0(z)}}{1-t\del_z m_0(z)e^{-m_0(z)}},\quad (\del_t h_t)(z_t(z))=-e^{-m_0(z)}.
\end{split}\end{align}
\end{proposition} 
\begin{proof}
The first three relations follow directly by taking derivative of \eqref{e:tmtrelation}. For the last relation, we have
\begin{align*}
\del_z((\del_t h_t)(z_t(z)))
=-(\del_t m_t)(z_t)\del_z z_t(z)=\del_z(-e^{-m_0(z)}).
\end{align*}
The last relation follows by noticing that $\lim_{z\rightarrow\infty}(\del_t h_t)(z_t(z))=0$.
\end{proof}

\subsection{Comparing with $\beta$-Dyson Brownian motion}\label{s:compareDBM}
In this section, we compare the central limit theorems of the $\beta$-nonintersecting random walks with those of the  $\beta$-Dyson Brownian motion. For general $\beta>0$, we recall the $\beta$-Dyson Brownian motion $\bmy(t)=(y_1(t), y_2(t),\cdots, y_n(t))$ is a diffusion process solving 
\begin{align}\label{e:DBM3}
\rd y_i(t)=\sqrt{\frac{2}{\beta n}}\rd \cB_i(t)+\frac{1}{n}\sum_{j\neq i}\frac{1}{y_i(t)-y_j(t)}\rd t +\rd t,\quad i=1,2,\cdots, n,
\end{align}
where $\{(\cB_1(t), \cB_2(t),\cdots, \cB_n(t))\}_{t\geq 0}$ are independent standard Brownian motions, and $\{\bmy(t)\}_{t>0}$ lives on the Weyl chamber $\bW^n=\{(\la_1,\la_2,\cdots,\la_n): \la_1>\la_2>\cdots>\la_n\}$. 
\begin{remark}
The expression \eqref{e:DBM3} is slightly different from \eqref{e:DBM1}. We add a constant drift term in \eqref{e:DBM3}, so that it matches with the dynamics of the $\beta$-nonintersecting Poisson random walks. 
\end{remark}
We denote the empirical measure process of \eqref{e:DBM3},  
\begin{align*}
\tilde \mu_t^n \deq \frac{1}{n}\sum_{i=1}^n \delta_{y_i(t)}.
\end{align*}
It follows from \cite{MR1217451,MR1176727}, if the initial empirical measure $\tilde \mu_0^n$ converges in the L{\'e}vy metric as $n$ goes to infinity towards a probability measure $\tilde \mu_0$ almost surely (in probability),
then, for any fixed time $T>0$, $\{\tilde \mu^n_t\}_{0\leq t\leq T}$ converges as $n$ goes to infinity  in $D([0,T], M_1(\bR))$ almost surely (in probability). The Stieltjes transform of the limit measure-valued process $\{\tilde\mu_t\}_{0\leq t\leq T}$ 
\begin{align*}
 \tilde m_t(z)=\int \frac{\rd \tilde \mu_t(x)}{x-z},\quad z\in \bC\setminus \bR,
\end{align*} 
is characterized by
\begin{align}\begin{split}\label{e:stDBM}
& \tilde m_t(\tilde z_t(z))=\tilde m_0(z),\\
& \tilde z_t(z)=z+t-t\tilde m_0(z),
\end{split}\end{align}
where $\tilde z_t(z)$ is well-defined on the domain
\begin{align*}
\tilde \Omega_t\deq\left\{z\in \bC\setminus \bR: \int \frac{\rd \tilde \mu_0(x)}{|x-z|^2}<\frac{1}{t}\right\}.
\end{align*}
We recall the limit empirical measure process $\mu_t$ of the $\beta$-nonintersecting Poisson random walks from Theorem \ref{t:LLN}, its Stieltjes transform $m_t(z)$ in \eqref{e:defmt}, and the key relations \eqref{e:deft} and \eqref{e:tmtrelation}. We also recall the Markov-Krein correspondence operator $Q$ from Theorem \ref{t:MKcor}. If we take the $\tilde \mu_0=Q(\mu_0)$, by the defining relation \eqref{e:defQ} of $Q$, we have
\begin{align*}
\tilde m_0(z)=1-e^{-m_0(z)}.
\end{align*}
Therefore, the characteristic lines for $m_t(z)$ and $\tilde m_t(z)$ are the same: 
\begin{align}\label{e:cline}
\tilde z_t(z)=z+t-t\tilde m_0(z)=z+te^{-m_0(z)}=z_t(z).
\end{align}
The Stieltjes transforms $m_t(z)$ and $\tilde m_t(z)$ satisfy
\begin{align*}
\tilde m_t(\tilde z_t(z))
=\tilde m_0(z)
=1-e^{-m_0(z)}
=1-e^{-m_t(z_t(z))}
=1-e^{-m_t(\tilde z_t(z))}.
\end{align*}
Since $\tilde z_t(z)$ is a surjection onto $\bC\setminus \bR$, we get $\tilde \mu_t=Q(\mu_t)$.

We denote the rescaled fluctuation process
\begin{align*}
\tilde g_t^n(z)=n(\tilde m_t^n(z)- \tilde m_t(z)).
\end{align*}
It follows from \cite{MR1819483,MR2418256}, if there exists a constant $\fa$, such that for any $z\in \bC\setminus\bR$, 
\begin{align*}
\bE\left[|\tilde g_0^n(z)|^2\right]\leq \fa(\Im[z])^{-2},
\end{align*}
and the random field $(\tilde g_0^n(z))_{z\in \bC\setminus\bR}$ weakly converges to a deterministic field $(\tilde g_0(z))_{z\in \bC\setminus \bR}$, in the sense of finite dimensional distributions, then, 
for any fixed time $T>0$,  the process $\{(\tilde g^n_t(\tilde z_t(z)))_{z\in\tilde \Omega_T}\}_{0\leq t\leq T}$ converges weakly towards a Gaussian process $\{(\tilde g_t(\tilde z_t(z)))_{z\in \tilde \Omega_T}\}_{0\leq t\leq T}$, in the sense of finite dimensional processes, with initial data $(\tilde g_0(z))_{z\in \bC\setminus\bR}$, means and covariances
\begin{align}\begin{split}\label{e:meanandvar2}
\bE[\tilde g_t(\tilde z_t(z))]&=\tilde \mu(t, z)=\frac{\tilde g_0(z)}{ \del_z\tilde z_t(z)}+\left(\frac{1}{2}-\frac{1}{2\theta}\right)\frac{\del_z^2\tilde z_t(z)}{(\del_z \tilde z_t(z))^2},\\
\cov[\tilde g_s(\tilde z_s(z)), \tilde g_t(\tilde z_t(z'))]&=\tilde \sigma(s, z, t, z')=\frac{1}{\theta}\frac{1}{\del_z\tilde z_s(z)\del_z\tilde z_t(z')}\left(\frac{1}{(z-z')^2}-\frac{\del_z \tilde z_{s\wedge t}(z)\del_z \tilde z_{s\wedge t}(z')}{(\tilde z_{s\wedge t}(z)-\tilde z_{s\wedge t}(z'))^2}\right),\\
\cov[\tilde g_s(\tilde z_s(z)), \overline{\tilde g_t(\tilde z_t(z'))}]&=\tilde \sigma(s, z,t,\bar z').
\end{split}
\end{align}
\begin{remark}
The statements in  \cite{MR1819483,MR2418256} are for $\beta$-Dyson Brownian motions with quadratic potential, which differ from \eqref{e:DBM3} by a rescaling of time and space. \eqref{e:meanandvar2} follows from \cite[Theorem 2.3]{MR2418256} by a change of variable.
\end{remark}

By comparing \eqref{e:meanandvar2} with \eqref{e:meanandvar},  if we replace the characteristic lines $\tilde z_t(z)$ in the expressions of the means and covariances of the random field $\{(\tilde g_t(\tilde z_t(z)))_{z\in\tilde \Omega_T}\}_{0\leq t\leq T}$ by $z_t(z)$, we get the means and variances of the random field  $\{(g_t( z_t(z)))_{z\in \Omega_T}\}_{0\leq t\leq T}$. If we take the $\tilde \mu_0=Q(\mu_0)$ and $\tilde g_0(z)=g_0(z)$, then from the discussion above, we have $\tilde \mu_t=Q(\mu_t)$ and $\tilde z_t(z)=z_t(z)$ from \eqref{e:cline}. Thus, 
\begin{align*}
\{(\tilde g_t(\tilde z_t(z)))_{z\in\tilde \Omega_T}\}_{0\leq t\leq T}\stackrel{d}{=}\{(g_t( z_t(z)))_{z\in \Omega_T}\}_{0\leq t\leq T},
\end{align*}
in the sense of finite dimensional processes. From the discussion above, we have that the following diagram commutes 
\begin{center}
\begin{tikzcd}\label{e:commute}
(\mu_0, g_0(z)) \arrow[rrrrrrr, "\beta-\text{nonintersecting Poisson random walks}"] \arrow[d, "Q\otimes I"]
& & & & & &&(\mu_t, g_t(z_t(z)))\arrow[d, "Q\otimes I"]\\
(\tilde \mu_0, \tilde g_0(z)) \arrow[rrrrrrr, "\beta-\text{Dyson Brownian motions}"]
& & & & & &&(\tilde \mu_t, \tilde g_t(\tilde z_t(z))),
\end{tikzcd}
\end{center}
where $Q$ is the Markov-Krein correspondence from Theorem \ref{t:MKcor}, and $I$ is the identity map.

Similar to the $\beta$-nonintersecting Poisson random walks, the fluctuation of the rescaled empirical measure process $\{n(\tilde \mu_t^n-\tilde \mu_t)\}_{0\leq t\leq T}$ can be identified with the Gaussian Free Field $\fG$ on $\bH$ with zero boundary conditions. We introduce the height function $\tilde H_n:  \bR\times \bR_{\geq 0}\mapsto \bZ_{\geq 0}$:
\begin{align*}
\tilde H_n(y,t)=|\{1\leq i\leq n: y_i(t)\geq n y\}|.
\end{align*}
We define the map $z\mapsto y(z)+\ri t(z)$ from $\bH$ to $\bH$,
\begin{align*}
 (y(z),t(z))=\left(\frac{(\bar z\tilde m_0(z)-z\tilde m_0(\bar z))+(z-\bar z)}{\tilde m_0(z)-\tilde m_0(\bar z)}, \frac{z-\bar z}{\tilde m_0(z)-\tilde m_0(\bar z)}\right).
\end{align*}
We note that the expressions for $y(z)$ and $t(z)$ are invariant with respect to complex conjugate, so $y(z)$ and $t(z)$ are indeed real for any $z\in \bH$. Since $y(z)=z+t(z)(1-\tilde m_0(z))\in \bR$, the map $z\mapsto (y(z),t(z))$ is in fact a diffeomorphism from $\bH$ to its image. We define the pull back height function on $\bH$,
\begin{align}\label{e:defHn}
H_n(z)\deq H_n(y(z), t(z)),\quad z\in \bH.
\end{align}
Then we have $\sqrt{\pi\theta}\left(H_n(z)-\bE[H_n(z)]\right)\rightarrow \fG(z)$, as $n$ goes to infinity, in the sense of Corollary \ref{c:GFF}.

\subsection{Stochastic differential equation for the Stieltjes transform}
\label{s:sde}
In this section we derive a stochastic differential equation for the Stieltjes transform of the empirical measure process $\mu_t^n$.
\begin{proposition}\label{p:msde}
The Stieltjes transform of the empirical measure process satisfies the following stochastic differential equation
\begin{align}\begin{split}\label{e:msde}
m^n_t(z_t)=m^n_0(z)
+\int_0^t \del_z m^n_s(z_s) e^{-m_s(z_s)}\rd s&+\theta n\int_{0}^t\left(\prod_{j=1}^n\left(1+\frac{1}{n}\frac{1}{z_s-x_j(s)/\theta n}\right)\right.\\
&-\left.\prod_{j=1}^n\left(1+\frac{1}{n}\frac{1}{(z_s-1/\theta n)-x_j(s)/\theta n}\right)\right)\rd s+M^n_t(z),
\end{split}\end{align}
where $z_t$ is defined in \eqref{e:dif2} with $z_0=z$ and $M^n_t(z)$ is a Martingale starting at $0$, with quadratic variations,
\begin{align}\begin{split}\label{e:qv}
[ M^n(z), M^n(z)]_t=\sum_{0\leq s\leq t}(m^n_s(z_s)-m^n_{s-}(z_s))^2, \\ [ M^n(z), \overline{M^n(z)}]_t=\sum_{0\leq s\leq t}|m^n_s(z_s)-m^n_{s-}(z_s)|^2.
\end{split}
\end{align}
\end{proposition} 

\begin{remark}
We remark that the integrand in \eqref{e:msde} also appears in the Nekrasov's equation in \cite[Section 4]{MR3668648}, which is crucial for the proof of the central limit theorems of the discrete $\beta$ ensembles. We can view \eqref{e:msde} as a dynamical version of the Nekrasov's equation.
\end{remark}

The $\beta$-nonintersecting random walk $\bmx(t)$ is a continuous time Markov jump process. We recall its generator $\cL^n_\theta $ from \eqref{e:generator}. By It{\'o}'s formula,  
\begin{align}\label{e:ito}
M^n_t(z)\deq m^n_t(z_t)-m^n_0(z_t)-\int_0^t \del_z m^n_s(z_s)\del_t z_s\rd s-\int_{0}^t \cL^n_\theta m^n_s(z_s)\rd s,
\end{align}
is a martingale with quadratic variations given by \eqref{e:qv}.

To estimate the integrand $\cL_\theta^n m_s^n(z_s)$ in \eqref{e:ito}, we need some algebraic facts about the generator $\cL_\theta^n$.
\begin{lemma}\label{l:average}
For any $\theta\in \bR$, we have
\begin{align}\label{e:average}
\sum_{i=1}^{n} \prod_{j:j\neq i}\frac{x_i-x_j+\theta}{x_i-x_j}=n.
\end{align}
\end{lemma}
\begin{proof}
We can rewrite the left hand side of \eqref{e:average} in terms of the Vandermond determinant in variables $x_1,x_2,\cdots, x_n$,
\begin{align*}
\sum_{i=1}^{n} \prod_{j:j\neq i}\frac{x_i-x_j+\theta}{x_i-x_j}
=\sum_{i=1}^n \frac{V(\bmx + \theta \bme_i)}{V(\bmx)}.
\end{align*}
We notice that $\sum_{i=1}^n V(\bmx+\theta \bme_i)$ is a degree $n(n-1)/2$ polynomial in variables $x_1,x_2,\cdots,x_n$. More importantly, it is antisymmetric. Therefore, there exists a constant $C(\theta, n)$ depending on $\theta$ and $n$ such that 
\begin{align}\label{e:average2}
\sum_{i=1}^n V(\bmx+\theta \bme_i)=C(\theta, n)V(\bmx).
\end{align}
We conclude \eqref{e:average} from \eqref{e:average2} by comparing the coefficient of the term $x_1^{n-1}x_2^{n-2}\cdots x_{n-1}$.
\end{proof}

The following identity will be crucial for the derivation of the stochastic differential equation of the Stieltjes transforms of the empirical measure process $\mu_t$.
\begin{corollary}\label{c:average3}
For any $\theta\in \bR$, we have
\begin{align}\label{e:average3}
\sum_{i=1}^{n} \left(\prod_{j:j\neq i}\frac{x_i-x_j+\theta}{x_i-x_j}\right)\frac{1}{n}\frac{1}{x_i/\theta n-z}=1-\prod_{j=1}^n\left(1+\frac{1}{n}\frac{1}{z-x_j/\theta n}\right).
\end{align}
\end{corollary}
\begin{proof}
We use Lemma \ref{l:average} for the vector $(x_1,x_2,\cdots, x_n, \theta n z)$, 
\begin{align*}\begin{split}
n+1&=\sum_{i=1}^{n} \left(\prod_{j:j\neq i}\frac{x_i-x_j+\theta}{x_i-x_j}\right)\frac{x_i-\theta n z+\theta}{x_i-\theta n z}
+\prod_{j=1}^n\frac{\theta n z-x_j+\theta}{\theta n z-x_j}\\
&=\sum_{i=1}^{n} \left(\prod_{j:j\neq i}\frac{x_i-x_j+\theta}{x_i-x_j}\right)+\sum_{i=1}^{n} \left(\prod_{j:j\neq i}\frac{x_i-x_j+\theta}{x_i-x_j}\right)\frac{\theta}{x_i-\theta n z}
+\prod_{j=1}^n\left(1+\frac{\theta}{\theta n z-x_j}\right)\\
&=n+\sum_{i=1}^{n} \left(\prod_{j:j\neq i}\frac{x_i-x_j+\theta}{x_i-x_j}\right)\frac{\theta}{x_i-\theta n z}
+\prod_{j=1}^n\left(1+\frac{\theta}{\theta n z-x_j}\right).
\end{split}\end{align*}
The claim \eqref{e:average3} follows by rearranging. 
\end{proof}

\begin{remark}\label{r:proofMKcor}
We can use Corollary \ref{c:average3} to give a construction of the operator $Q$ in Theorem \ref{t:MKcor}. Let $\mu$ be a measure as in Theorem \ref{t:MKcor}. For any large integer $m>0$, we discretize $\mu$ on the scale $1/m$ and define
\begin{align*}
\mu^m\deq \frac{1}{m}\sum_{i=1}^m \delta_{y_i^m},\quad
\frac{i-1/2}{m}=\int_{y_i^m}^{\infty}\rd \mu(x), \quad 1\leq i\leq m.
\end{align*}
As $m$ goes to infinity, $\mu^m$ weakly converges to $\mu$. Since the density of $\mu$ is bounded by $1$, we have $y_{i}^m-y_{i+1}^m\geq 1/m$ for all $1\leq i\leq m-1$. The Perelomov-Popov measure is defined as
\begin{align*}
Q^m(\mu^m)\deq \frac{1}{m}\sum_{i=1}^m \prod_{j:j\neq i}\frac{y_i^m-y_{j}^m+1/m}{y_i^m-y_j^m}\delta_{y_i^m}.
\end{align*}
Since $y_{i}^m-y_{i+1}^m\geq 1/m$, $Q^m(\mu^m)$ is a positive measure. Moreover, thanks to Lemma \ref{l:average}, $Q^m(\mu^m)$ is a probability measure. We denote the Stieltjes transform of $\mu$, $\mu^m$ and $Q^m(\mu^m)$ by $m_{\mu}(z)$, $m_{\mu^m}(z)$ and $m_{Q^m(\mu^m)}(z)$ respectively. Since $\mu^m$ weakly converges to $\mu$ as $m$ goes to infinity, 
\begin{align*}
\lim_{m\rightarrow \infty}m_{\mu^m}(z)=m_\mu(z),\quad z\in \bC\setminus\bR.
\end{align*}
For the Stieltjes transform $m_{Q^m(\mu^m)}(z)$, we use Corollary \ref{c:average3},
\begin{align*}
m_{Q^m(\mu^m)}(z)
&=\frac{1}{m}\sum_{i=1}^m \prod_{j:j\neq i}\frac{y_i^m-y_{j}^m+1/m}{y_i^m-y_j^m}\frac{1}{y_i^m-z}\\
&=1-\prod_{j=1}^m\left(1-\frac{1}{m}\frac{1}{y_j^m-z}\right)
=1-\exp\left\{-\frac{1}{m}\sum_{j=1}^m\frac{1}{y_j^m-z}+\OO\left(\frac{1}{m}\right)\right\}\\
&=1-\exp\left\{-m_{\mu^m}(z)+\OO\left(\frac{1}{m}\right)\right\}\rightarrow 1-e^{-m_\mu(z)},
\end{align*}
as $m$ goes to infinity. Since 
\begin{align*}
\lim_{y\rightarrow \infty}\ri y\left(1-e^{-m_\mu(\ri y)}\right)=-1,
\end{align*}
by \cite[Theorem 1]{MR1962995}, there exists a probability measure $Q(\mu)$ with Stieltjes transform $1-e^{-m_\mu(z)}$, and $Q^m(\mu^m)$ weakly converges to $Q(\mu)$.
\end{remark}

\begin{proof}[Proof of Proposition \ref{p:msde}]

For the integrand $\cL^n_\theta m^n_s(z)$, we have
\begin{align*}\begin{split}
\cL^n_\theta m^n_s(z)
&=\theta n\sum_{i=1}^{n}\left(\prod_{j:j\neq i}\frac{x_i(s)-x_j(s)+\theta}{x_i(s)-x_j(s)}\right)\left(\frac{1}{n}\frac{1}{(x_i(s)+1)/\theta n-z_s}-\frac{1}{n}\frac{1}{x_i(s)/\theta n-z_s}\right)\\
&=\theta n\left(\prod_{j=1}^n\left(1+\frac{1}{n}\frac{1}{z_s-x_j(s)/\theta n}\right)
-\prod_{j=1}^n\left(1+\frac{1}{n}\frac{1}{(z_s-1/\theta n)-x_j(s)/\theta n}\right)\right).
\end{split}\end{align*}
where we used Corollary \ref{c:average3}. Combining with \eqref{e:ito}, this finishes the proof Proposition \ref{p:msde}
\end{proof}

\subsection{Law of large numbers}\label{s:LLN}
In this section we analyze \eqref{e:msde}, and prove the Law of large numbers for the $\beta$-nonintersecting Poisson random walk. 

We define an auxiliary process $N^n_t$, which counts the number of jumps for the $\beta$-nonintersecting Poisson random walk $\bmx(t)$,
\begin{align}\label{e:defNt}
N^n_t=\sum_{i=1}^n \left(x_i(t)-x_i(0)\right).
\end{align}
The Poisson process $N^n_t$ will be used later to control the martingale term $M^n_t(z)$ in \eqref{e:msde}.

\begin{proposition}
$N^n_t$ is a Poisson process, starting at $0$,  with jump rate $\theta n^2$.
\end{proposition}
\begin{proof}
According to the generator \eqref{e:generator} of the $\beta$-nonintersecting Poisson random walk, the process $N^n_t$ increases $1$ with rate
\begin{align*}
\theta n\sum_{i=1}^n\sum_{j:j\neq i}\frac{x_i-x_j+\theta}{x_i-x_j}=\theta n^2,
\end{align*}
where we used Proposition \ref{e:average}.
\end{proof}

In the following we simplify the stochastic differential equation of $m_t(z_t)$,  i.e. the second integrand in \eqref{e:msde}.  By Proposition \ref{p:ztp}, for any $z\in \Omega_T$, and $0\leq t\leq T$, we have $|\Im[z_t]|\geq |\Im[z_T]|>0$. Therefore, we have the trivial bound $1/|x_j(s)/\theta n-z_t|=\OO(1)$, where the implicit constant depends on $\Im[z_T]$.

For the first term in the second integrand in \eqref{e:msde}, by the Tylor expansion
\begin{align}\begin{split}\label{e:t1}
&\phantom{{}={}}\prod_{j=1}^n\left(1+\frac{1}{n}\frac{1}{z_s-x_j(s)/\theta n}\right)
=\exp\left(\sum_{j=1}^n \ln\left(1+\frac{1}{n}\frac{1}{z_s-x_j(s)/\theta n}\right)\right)\\
&=\exp\left(\sum_{j=1}^n \frac{1}{n}\frac{1}{z_s-x_j(s)/\theta n}-\frac{1}{2n^2}\frac{1}{(z_s-x_j(s)/\theta n)^2}+\frac{1}{3n^3}\frac{1}{(z_s-x_j(s)/\theta n)^3}+\OO\left(\frac{1}{n^4}\right)\right)\\
&=e^{-m^n_s(z_s)}\left(1-\frac{1}{2}\frac{\del_z m^n_s(z_s)}{n}+\frac{1}{8}\frac{(\del_z m^n_s(z_s))^2}{n^2}-\frac{1}{6}\frac{\del_z^2 m^n_s(z_s)}{n^2}+\OO\left(\frac{1}{n^3}\right)\right).
\end{split}\end{align}
Similarly, for the second term in the integrand,
\begin{align}\begin{split}\label{e:t2}
&\phantom{{}={}}\prod_{j=1}^n\left(1+\frac{1}{n}\frac{1}{(z_s-1/\theta n)-x_j(s)/\theta n}\right)
=\exp\left(\sum_{j=1}^n \ln\left(1+\frac{1}{n}\frac{1}{(z_s-1/\theta n)-x_j(s)/\theta n}\right)\right)\\
&=e^{-m^n_s(z)}\left(1-\left(\frac{1}{2}-\frac{1}{\theta}\right)\frac{\del_z m^n_s(z)}{n}+\frac{1}{2}\left(\frac{1}{\theta}-\frac{1}{2}\right)^2\frac{(\del_z m^n_s(z_s))^2}{n^2}-\right.\\
&\phantom{{}=e^{-m^n_s(z_s)}\left(1-\left(\frac{1}{2}-\frac{1}{\theta}\right)\frac{\del_z m^n_s(z_s)}{n}\right.}\left.-\frac{1}{2}\left(\frac{1}{\theta^2}-\frac{1}{\theta}+\frac{1}{3}\right)\frac{\del_z^2 m^n_s(z_s)}{n^2}+\OO\left(\frac{1}{n^3}\right)\right).
\end{split}\end{align}
The difference of \eqref{e:t1} and \eqref{e:t2} is 
\begin{align}\begin{split}\label{e:t3}
&\phantom{{}={}}\prod_{j=1}^n\left(1+\frac{1}{n}\frac{1}{(z_s-1/\theta n)-x_j(s)/\theta n}\right)-\prod_{j=1}^n\left(1+\frac{1}{n}\frac{1}{(z_s-1/\theta n)-x_j(s)/\theta n}\right)\\
&=e^{-m^n_s(z_s)}\left(-\frac{\del_z m^n_s(z_s)}{\theta n}+\left(\frac{1}{2}-\frac{1}{2\theta}\right)\frac{(\del_z m^n_s(z_s))^2-\del_z^2 m^n_s(z_s)}{\theta n^2}+\OO\left(\frac{1}{n^3}\right)\right).
\end{split}\end{align}
We can use \eqref{e:t3} to simplify the stochastic differential equation \eqref{e:msde}, 
\begin{align}\begin{split}\label{e:newsde}
m^n_t(z_t)=m^n_0(z)
&+\int_0^t \del_z m^n_s(z_s) \left(e^{-m_s(z_s)}-e^{-m^n_s(z_s)}\right)\rd s+\\
&+\left(\frac{1}{2}-\frac{1}{2\theta}\right)\int_{0}^t\frac{((\del_z m^n_s(z_s))^2-\del_z^2 m^n_s(z_s))e^{-m_s^n(z_s)}}{n}\rd s+M^n_t(z)+\OO\left(\frac{t}{n^2}\right),
\end{split}\end{align}
where the implicit constant depends on $\theta$ and $\Im[z_T]$.

In the following we estimate the martingale $M^n_t(z)$ using the Burkholder-Davis-Gundy inequality. For the quadratic variation of $M^n_t(z)$, we have
\begin{align}\begin{split}\label{e:quadvar}
&\phantom{{}={}}[ M^n(z), \overline{M^n(z)}]_t
=\sum_{0< s\leq t} |m^n_s(z_s)-m^n_{s-}(z_s)|^2\\
&=\frac{1}{n^2}\sum_{0< s\leq t\atop \Delta x_i(s)>0}\left|\frac{1}{x_i(s)/\theta n-z_s}-\frac{1}{x_i(s-)/\theta n-z_s}\right|^2
= \OO\left(\frac{N^n_t}{ n^4}\right),
\end{split}\end{align}
where the implicit constant depends on $\theta$ and $\Im[z_T]$.
It follows from the Burkholder-Davis-Gundy inequality, for any $p\geq 1$, we have
\begin{align}\begin{split}\label{e:BDG}
\bE\left[\left(\sup_{0\leq t\leq T}|M^n_t(z)|\right)^p\right]^{1/p}
&\leq  Cp \bE\left[[ M^n(z), \overline{M^n(z) }]_T^{p/2}\right]^{1/p}
\\
&=\OO\left(\frac{p}{ n^2}\bE\left[\left(N^n_T\right)^{p/2}\right]^{1/p} \right)=\OO\left(\frac{T^{1/2}p^{3/2}}{n}\right).
\end{split}\end{align}
By the Markov's  inequality, we have 
\begin{align*}
\bP\left(\sup_{0\leq t\leq T}|M^n_t(z)|\geq \varepsilon\right)\leq \left(\frac{CT^{1/2}p^{3/2}}{\varepsilon n}\right)^p.
\end{align*}
Therefore it follows by taking $p>1$, $\sup_{0\leq t\leq T}|M^n_t(z)|$ converges to zero almost surely as $n$ goes to infinity.


For any $0\leq t\leq T$, we can rewrite \eqref{e:newsde} as
\begin{align*}\begin{split}
|m^n_t(z_t)-m^n_0(z)|
&\leq |m^n_0(z)-m_0(z)|+C\left(\int_0^t |m_s(z_s)-m_0(z)|\rd s+\frac{t}{n}\right)+\sup_{0\leq t\leq T}|M_t^n|\\
&=C\left(\int_0^t |m_s(z_s)-m_0(z)|\rd s\right)+|m^n_0(z)-m_0(z)|+\oo(1)
\end{split}\end{align*}
where the constant $C$ depends on $\theta$, $T$ and $\Im[z_T]$, and the term $\oo(1)$ converges to zero almost surely and is uniform for $0\leq t\leq T$. Thus, it follows from Gronwell's inequality,
\begin{align*}
\sup_{0\leq t\leq T}|m_t(z_t)-m_0(z)|\leq C|m_0^n(z)-m_0(z)|+o(1),
\end{align*}
which converges to zero almost surely (in probability), if $|m_0^n(z)-m_0(z)|$ converges to zero almost surely (in probability). This finishes the proof of Theorem \ref{t:LLN}.

\section{Central limit theorems for the empirical measure process}

In this section, we prove the central limit theorems for the rescaled empirical measure process $\{n(\mu^n_t-\mu_t)\}_{0\leq t \leq T}$ with analytic test functions. 
\subsection{Central limit theorems}\label{s:CLT}
Theorem \ref{t:CLT} follows from the following proposition.
\begin{proposition}\label{p:CLT}
We assume Assumption \ref{a:initiallaw}. Then for any values $z_1,z_2,\cdots, z_m\in \bC\setminus\bR$ and time $T<\min\{\ft(z_1),\ft(z_2),\cdots, \ft(z_m)\}$ as defined in \eqref{e:deft}, the random processes $\{(g^n_t(z_t(z_j)))_{1\leq j\leq m}\}_{0\leq t\leq T}$ converge weakly in the Skorokhod space $D([0,T], \bC^m)$ towards a Gaussian process $\{(\mathcal{G}_j(t))_{1\leq j\leq m}\}_{0\leq t\leq T}$, which is the unique solution of the system of stochastic differential equations
\begin{align}\begin{split}\label{e:limitprocess}
\cG_j(t)&=\cG_j(0)+ \int_0^t \frac{\del_zm_0(z_j)e^{-m_0(z_j)}}{1-s\del_zm_0(z_j)e^{-m_0(z_j)}} \cG_j(s)\rd s\\
&+\left(\frac{1}{2}-\frac{1}{2\theta}\right)\int_0^t\frac{((\del_z m_0(z_j))^2-\del_z^2 m_0(z_j))e^{-m_0(z_j)}}{(1-s\del_z m_0(z_j)e^{-m_0(z_j)})^3}\rd s+\cW_j(t),\quad 1\leq j\leq m,
\end{split}\end{align}
with initial data $(g_0(z_j))_{1\leq j\leq m}$ given in Assumption \ref{a:initiallaw}, and $\{(\cW_j(t))_{1\leq j\leq m}\}_{0\leq t\leq T}$ is a centered Gaussian process independent of $(\cG_j(0))_{1\leq j\leq m}$, and 
\begin{align}\begin{split}\label{e:covW1}
\langle \cW_j, \cW_k\rangle_t=-\frac{1}{\theta}\int_0^t\frac{(\del_s m_s)(z_s(z_j))+(\del_sm_s)(z_s(z_k))}{(z_s(z_j)-z_s(z_k))^2}-\frac{2(e^{-m_0(z_j)}-e^{-m_0(z_k)})}{(z_s(z_j)-z_s(z_k))^3}\rd s,\\
\langle\cW_j, \bar{\cW}_k\rangle_t=-\frac{1}{\theta}\int_0^t\frac{(\del_s m_s)(z_s(z_j))+(\del_sm_s)(z_s(\bar{z}_k))}{(z_s(z_j)-z_s(\bar{z}_k))^2}-\frac{2(e^{-m_0(z_j)}-e^{-m_0(\bar{z}_k)})}{(z_s(z_j)-z_s(\bar{z}_k))^3}\rd s,
\end{split}\end{align}
where
\begin{align}\begin{split}\label{e:covW2}
\langle \cW_j, \cW_j\rangle_t
&=\lim_{z_k\rightarrow z_j}-\frac{1}{\theta}\int_0^t\frac{(\del_s m_s)(z_s(z_j))+(\del_sm_s)(z_s(z_k))}{(z_s(z_j)-z_s(z_k))^2}-\frac{2(e^{-m_0(z_j)}-e^{-m_0(z_k)})}{(z_s(z_j)-z_s(z_k))^3}\rd s\\
&=-\frac{1}{6\theta }\int_0^t (\del_s\del^2_z m_s)(z_s(z_j))\rd s.
\end{split}\end{align}

\end{proposition}

\begin{proof}[Proof of Theorem \ref{t:CLT}]
We can solve for $\{(\cG_j(t))_{1\leq j\leq m}\}_{0\leq t\leq T}$ explicitly. From \eqref{e:limitprocess}, we have
\begin{align}\begin{split}\label{e:rearrange}
\rd(1-t\del_zm_0(z_j)e^{-m_0(z_j)})\cG_j(t)&=\left(\frac{1}{2}-\frac{1}{2\theta}\right)\frac{((\del_z m_0(z_j))^2-\del_z^2 m_0(z_j))e^{-m_0(z_j)}}{(1-t\del_z m_0(z_j)e^{-m_0(z_j)})^2}\rd t\\
&+(1-t\del_zm_0(z_j)e^{-m_0(z_j)})\rd\cW_j(t).
\end{split}\end{align}
We integrate both sides of \eqref{e:rearrange},
\begin{align*}
\cG_j(t)=\frac{\cG_j(0)}{1-t\del_zm_0(z_j)e^{-m_0(z_j)}}
+\left(\frac{1}{2}-\frac{1}{2\theta}\right)\frac{t((\del_z m_0(z_j))^2-\del_z^2 m_0(z_j))e^{-m_0(z_j)}}{(1-t\del_z m_0(z_j)e^{-m_0(z_j)})^2}
+ \cB_j(t),
\end{align*}
where 
\begin{align*}
\cB_j(t)=\frac{1}{(1-t\del_zm_0(z_j)e^{-m_0(z_j)})}\int_0^t (1-s\del_zm_0(z_j)e^{-m_0(z_j)})\rd\cW_j(s).
\end{align*}
By a straightforward (but tedious and lengthy) calculation, using \eqref{e:mtestimate}, \eqref{e:covW1} and \eqref{e:covW2}, we get the covariances of $\{(\cB_j(t))_{1\leq j\leq m}\}_{0\leq t\leq T}$,
\begin{align*}
\cov[ \cB_j(s), \cB_k(t)]
&=\frac{\int_0^{s\wedge t} (1-u\del_zm_0(z_j)e^{-m_0(z_j)})(1-u\del_zm_0(z_k)e^{-m_0(z_k)})\rd\langle\cW_j, \cW_k\rangle_u}{(1-s\del_zm_0(z_j)e^{-m_0(z_j)})(1-t\del_zm_0(z_k)e^{-m_0(z_k)})}\\
&=\frac{1}{\theta}\frac{1}{ (1-s\del_zm_0(z_j)e^{-m_0(z_j)})(1-t\del_zm_0(z_k)e^{-m_0(z_k)})}\\
&\times\left(\frac{1}{(z_j-z_k)^2}-\frac{(1-(s\wedge t)\del_zm_0(z_j)e^{-m_0(z_j)})(1-(s\wedge t)\del_zm_0(z_k)e^{-m_0(z_k)})}{(z_j-z_k+(s\wedge t)(e^{-m_0(z_j)}-e^{-m_0(z_k)}))^2}\right)=\sigma(s, z_j,t,z_k)\\
\cov[ \cB_j(s), \overline{\cB_k(t)}]&=\sigma(s,z_j,t, \bar z_k),
\end{align*}
and
\begin{align*}
&\phantom{{}={}}\cov[\cB_j(s), \cB_j(t)]=\frac{\int_0^{s\wedge t} (1-u\del_zm_0(z_j)e^{-m_0(z_j)})^2\rd\langle\cW_j, \cW_j\rangle_u}{(1-s\del_zm_0(z_j)e^{-m_0(z_j)})(1-t\del_zm_0(z_j)e^{-m_0(z_j)})}\\
&=\frac{(s\wedge t)e^{-m_0(z_j)}(2(\del_z m_0(z_j))^3-6\del_z m_0(z_j)\del_z^2 m_0(z_j)+2\del_z^3m_0(z_j))}{12\theta(1-(s\wedge t)\del_z m_0(z_j)e^{-m_0(z_j)})^3(1-(s\vee t)\del_z m_0(z_j)e^{-m_0(z_j)})}\\
&+\frac{(s\wedge t)^2e^{-2m_0(z_j)}((\del_z m_0(z_j))^4+3(\del_z^2 m_0(z_j))^2-2\del_z m_0(z_j)\del_z^3m_0(z_j))}{12\theta(1-(s\wedge t)\del_z m_0(z_j)e^{-m_0(z_j)})^3(1-(s\vee t)\del_z m_0(z_j)e^{-m_0(z_j)})}=\sigma(s,z_j,t, z_j)=\lim_{z_k\rightarrow z_j}\sigma(s, z_j,t, z_k).
\end{align*}
This finishes the proof of Theorem \ref{t:CLT}.
\end{proof}

We divide the proof of Proposition \ref{p:CLT} into three steps. In Step one we prove  the tightness of the processes $\{(g^n_t(z_t(z_j)))_{1\leq j\leq m})\}_{0\leq t\leq T}$ as $n$ goes to infinity. In Step two, we prove that the martingale term $\{(nM_t^n(z_j))_{1\leq j\leq m}\}_{0\leq t\leq T}$ converges weakly to a centered complex Gaussian process. In Step three, we prove that the subsequential limits of $\{(g^n_t(z_t(z_j)))_{1\leq j\leq m})\}_{0\leq t\leq T}$ solve the stochastic differential equation \eqref{e:limitprocess}. Proposition \ref{p:CLT} follows from this fact and the uniqueness of the solution to \eqref{e:limitprocess}.

\begin{proof}[Proof of Proposition \ref{p:CLT}]

{\emph{Step one: tightness.}}

We first prove the tightness of the martingale term.
\begin{claim}
We assume the assumptions of Proposition \ref{p:CLT}. Then as $n$ goes to infinity, the random processes $\{(nM^n_t(z_j))_{1\leq j\leq m}\}_{0\leq t\leq T}$, and $\{(n^2[ M^n(z_j), M^n(z_k)]_t)_{1\leq j,k\leq m}\}_{0\leq t\leq T}$ are tight.

\end{claim}
We apply the sufficient condition for tightness of \cite[Chapter 6, Proposition 3.26]{MR1943877}. We need to check the modulus conditions: for any $\varepsilon>0$ there exists a $\delta>0$ such that 
\begin{align}\label{e:modulus1}
&\bP\left(\sup_{1\leq j\leq m}\sup_{0\leq t\leq t'\leq T, t'-t\leq \delta}|n(M^n_{t'}(z_j)-M^n_t(z_j))|\geq \varepsilon\right)\leq \varepsilon,\\
\label{e:modulus2}&\bP\left(\sup_{1\leq j,k\leq m}\sup_{0\leq t\leq t'\leq T, t'-t\leq \delta}\left|n^2([ M^n(z_j), M^n(z_k)]_{t'}-[ M^n(z_j), M^n(z_k)]_t)\right|\geq \varepsilon\right)\leq \varepsilon.
\end{align}

For \eqref{e:modulus1}, since $\{M^n_{t'}(z_j)-M^n_t(z_j)\}_{t\leq t'\leq T\vee t+\delta}$ is a martingale, it follows from the Burkholder-Davis-Gundy inequality, for any $p\geq 1$, we have
\begin{align*}\begin{split}
&\phantom{{}={}}\bE\left[\left(\sup_{t\leq t'\leq T\vee t+\delta}|n(M^n_{t'}(z_j)-M^n_t(z_j))|\right)^p\right]^{1/p}\\
&\leq  Cpn \bE\left[[ M^n(z_j)-M^n_t(z_j), \overline{M^n(z_j)-M_t^n(z_j) }]_{T\vee t+\delta}^{p/2}\right]^{1/p}
\\
&=\OO\left(\frac{p}{ n}\bE\left[\left(N^n_{T\vee t+\delta}-N^n_t\right)^{p/2}\right]^{1/p} \right)=\OO\left(\delta^{1/2}p^{3/2}\right).
\end{split}\end{align*}
where the implicit constant depends on $\theta$ and $\min_{1\leq j\leq m}|\Im[z_T(z_j)]|$. By the Markov's  inequality, we have 
\begin{align*}
\bP\left(\sup_{t\leq t'\leq T\vee t+\delta}|n(M^n_{t'}(z_j)-M^n_t(z_j))|\geq \varepsilon\right)\leq \left(\frac{C\delta^{1/2}p^{3/2}}{\varepsilon }\right)^p.
\end{align*}
Let $t_k=(k-1)\delta\vee T$ for $1\leq k\leq \lfloor1/\delta\rfloor$. By a union bound
\begin{align*}\begin{split}
&\phantom{{}={}}\bP\left(\sup_{1\leq j\leq m}\sup_{0\leq t\leq t'\leq T, t'-t\leq \delta}|n(M^n_{t'}(z_j)-M^n_t(z_j))|\geq \varepsilon\right)\\
&\leq \frac{m}{\delta} \sup_{1\leq j\leq m}\sup_{1\leq k\leq \lfloor 1/\delta \rfloor}\bP\left(\sup_{t_k\leq t\leq T\vee(t_k+\delta)}|n(M^n_{t}(z_j)-M^n_{t_k}(z_j))|\geq \varepsilon/2\right)\\
&\leq \frac{m}{\delta}\left(\frac{2C\delta^{1/2}p^{3/2}}{\varepsilon}\right)^p\leq \varepsilon.
%
\end{split}\end{align*}
if we take $p>2$ and $\delta$ small enough. This finishes the proof of \eqref{e:modulus1}.

The modulus of the process $n^2[M^n(z_j), M^n(z_k)]_{t}$ is dominated by the Poisson process $n^{-2}N_t^n$ in the following sense. For any  $0\leq t\leq T$ and $t\leq t'\leq T\vee(t+\delta)$,
\begin{align}\begin{split}\label{e:control}
&\phantom{{}={}}\left|n^2([ M^n(z_j), M^n(z_k)]_{t'}-[ M^n(z_j), M^n(z_k)]_t)\right|\\
&\leq n^2\sum_{t< s\leq t'} |m^n_s(z_s(z_j))-m^n_{s-}(z_s(z_j))||m^n_s(z_s(z_k))-m^n_{s-}(z_s(z_k))|\\
&=\sum_{t< s\leq t'\atop \Delta x_i(s)>0}\left|\frac{1}{x_i(s)/\theta n-z_s(z_j)}-\frac{1}{x_i(s-)/\theta n-z_s(z_j)}\right|\left|\frac{1}{x_i(s)/\theta n-z_s(z_k)}-\frac{1}{x_i(s-)/\theta n-z_s(z_k)}\right|\\
&= \OO\left(\frac{N^n_{t'}-N^n_t}{ n^2}\right),
\end{split}\end{align}
where the implicit constant depends on $\theta$ and $\min_{1\leq j\leq m}|\Im[z_T(z_j)]|$.
By the same argument as for \eqref{e:modulus1}, we have that $n^{-2}N_t^n$ satisfies the modulus condition:
\begin{align}
\label{e:modulus3}\bP\left(\sup_{0\leq t\leq t'\leq T, t'-t\leq \delta}\left|\frac{N_{t'}^n-N_t^n}{n^2}\right|\geq \varepsilon\right)\leq \varepsilon.
\end{align}
The claim \eqref{e:modulus2} follows from combining \eqref{e:control} and \eqref{e:modulus3}. 
%

\begin{claim}
We assume the assumptions of Proposition \ref{p:CLT}. Then as $n$ goes to infinity, the random processes $\{(g^n_t(z_t(z_j)))_{1\leq j\leq m})\}_{0\leq t\leq T}$ are tight.
\end{claim}
We apply the sufficient condition for tightness of \cite[Chapter 6, Proposition 3.26]{MR1943877}, and check the modulus condition: for any $\varepsilon>0$ there exists a $\delta>0$ such that 
\begin{align}\label{e:modulus}
\sup_{n\geq 1}\bP\left(\sup_{1\leq i\leq m}\sup_{0\leq t\leq t'\leq T, t'-t\leq \delta}|g^n_{t'}(z_{t'}(z_j))-g^n_t(z_t(z_j))|\geq \varepsilon\right)\leq \varepsilon.
\end{align}

Before we prove the modulus condition \eqref{e:modulus}, we first prove that as $n$ goes to infinity, the random processes $\{(g_t^n(z_t(z_j)))_{1\leq j\leq m}\}_{0\leq t\leq T}$ are stochastically bounded, i.e. for any $\varepsilon>0$, there exists $M>0$ such that 
\begin{align}\label{e:stBound}
\sup_{n\geq 1}\bP\left(\sup_{1\leq j\leq m}\sup_{0\leq t\leq T}|g^n_t(z_t(z_j))|\geq M\right)\leq \varepsilon.
\end{align}
By rearranging \eqref{e:newsde}, for any $1\leq i \leq m$, we get
\begin{align}\begin{split}\label{e:gsde1}
g_t^n(z_t(z_j))&=g_0^n(z_j)+\int_0^t \del_z m^n_s(z_s(z_j))e^{-m_0(z_j)} n\left(1-e^{-(m^n_s(z_s(z_j))-m_0(z_j))}\right)\rd s+\\
&+\left(\frac{1}{2}-\frac{1}{2\theta}\right)\int_{0}^t\left((\del_z m^n_s(z_s(z_j)))^2-\del_z^2 m^n_s(z_s(z_j))\right)e^{-m_s^n(z_s(z_j))}\rd s+nM^n_t(z_j)+\OO\left(\frac{t}{n}\right).
\end{split}\end{align}
For the second and third terms on the righthand side of \eqref{e:gsde1}, we have 
\begin{align}\begin{split}\label{e:bound1}
\left|\del_z m^n_s(z_s(z_j))e^{-m_0(z_j)} n\left(1-e^{-(m^n_s(z_s(z_j))-m_0(z_j))}\right)\right|&\leq C|g_s^n(z_s(z_j))|,\\
\left|(\del_z m^n_s(z_s(z_j)))^2-\del_z^2 m^n_s(z_s(z_j))e^{-m_s^n(z_s(z_j))}\right|&\leq C,
\end{split}
\end{align}
where the constants $C$ depends on $\min_{1\leq j\leq m}|\Im[z_T(z_j)]|$. From \eqref{e:BDG}, we have
\begin{align}\begin{split}\label{e:bound2}
\bE\left[\left(\sup_{0\leq t\leq T}|nM^n_t(z_j)|\right)^p\right]^{1/p}
&\leq  CT^{1/2}p^{3/2},
\end{split}\end{align}
where the constant $C$ depends on $\theta$ and $\min_{1\leq j\leq m}|\Im[z_T(z_j)]|$. Combining \eqref{e:bound1}, \eqref{e:bound2} with the Gronwall's inequality, we get that the process $\{(g_t^n(z_t(z_j))_{1\leq j\leq m}\}_{1\leq t\leq T}$ is stochastically bounded \eqref{e:stBound}.

On the event $\{\sup_{1\leq j\leq m}\sup_{0\leq t\leq T}|g^n_t(z_t(z_j))|\leq M\}$, for any $0\leq t\leq T$ and $t\leq t'\leq T\vee t+\delta$, we have
\begin{align*}
|g^n_{t'}(z_{t'}(z_j))-g^n_t(z_t(z_j))|\leq C(M+1)\delta+n(M_{t'}^n(z_j)-M_t^n(z_j))+\OO\left(\frac{t}{n}\right).
\end{align*}
The claim \eqref{e:modulus} follows from the tightness of $\{(nM_t^n(z_j))_{1\leq j\leq m}\}_{0\leq t\leq T}$.

\emph{ Step two: weak convergence of the martingale term.}

We define a sequence of stopping times $\tau_0^n, \tau^n_1, \tau^n_2, \tau^n_3, \cdots$, where $\tau_0^n=0$ and for $l\geq 1$, $\tau^n_l$  is the time of the $l$-th jump of the Poisson process $N_t^n$. The following estimate follows from the tail estimate of the exponential random variables.

\begin{claim}\label{c:waittime}
Fix $\theta>0$ and time $T>0$. For any $\varepsilon>0$, there exists a $M>0$ such that 
\begin{align}\label{e:waittime}
\sup_{n\geq1}\bP\left(\sup_{0<\tau^n_j\leq t} |\tau^n_j-\tau^n_{j-1}|\geq \frac{M\ln n}{n^2}\right)\leq \varepsilon
\end{align}
\end{claim}
Since the waiting time of $N_t^n$ is an exponential random variable of rate $\theta n^2$, for any $j\geq 1$, we have
\begin{align}\label{e:waittime2}
\bP\left(|\tau^n_j-\tau^n_{j-1}|\geq \frac{M\ln n}{n^2}\right)=\exp\left\{-\theta n^2\frac{M\ln n}{n^2}\right\}=n^{-\theta M}.
\end{align}
The claim \eqref{e:waittime} follows from \eqref{e:waittime2} and a union bound.
\begin{claim}\label{c:Mconverge}
We assume the assumptions of Proposition \ref{p:CLT}. Then as $n$ goes to infinity, the complex martingales $\{(nM_t^n(z_j))_{1\leq j\leq m}\}_{0\leq t\leq T}$ converge weakly in $D([0,T], \bC^m)$ towards a centered complex Gaussian process $\{(\cW_j(t))_{1\leq j\leq m}\}_{0\leq t\leq T}$, with quadratic variation given by \eqref{e:covW1} and \eqref{e:covW2}.
\end{claim}
We notice that $\overline{M_t^n(z_j)}=M_t^n(\bar{z}_j)$. Claim \ref{c:Mconverge} follows from \cite[Chapter 7,  Theorem 1.4]{MR838085} and the weak convergence of the quadratic variations,
\begin{align}
\label{e:var}n^2[ M^n(z_j), M^n(z_j) ]_t
&\Rightarrow -\frac{1}{6\theta }\int_0^t (\del_s\del^2_z m_s)(z_s(z_j))\rd s, \\
\label{e:cov}n^2[ M^n(z_j), M^n(z_k) ]_t
&\Rightarrow-\frac{1}{\theta}\int_0^t\frac{(\del_s m_s)(z_s(z_j))+(\del_sm_s)(z_s(z_k))}{(z_s(z_j)-z_s(z_k))^2}-\frac{2(e^{-m_0(z_j)}-e^{-m_0(z_k)})}{(z_s(z_j)-z_s(z_k))^3}\rd s.
\end{align}
Thanks to \eqref{e:modulus2}, we know that the processes $\{(n^2[M^n_t(z_j), M^n_t(z_k)]_t)_{1\leq j,k\leq m}\}_{0\leq t\leq T}$ are tight. For \eqref{e:var} and \eqref{e:cov}, it remains to prove the weak convergence of any fixed time.

By definition, the quadratic variation $n^2[M^n_t(z_j), M^n_t(z_k)]_t$ is given by
\begin{align}\begin{split}\label{e:quadvar2}
n^2[ M^n(z_j), M^n(z_k) ]_t
&=
\sum_{0< s\leq t\atop \Delta x_i(s)>0}\left(\frac{1}{x_i(s)/\theta n-z_s(z_j)}-\frac{1}{x_i(s-)/\theta n-z_s(z_j)}\right)\\
&\phantom{{}=
\sum_{0< s\leq t\atop \Delta x_i(s)>0}}\left(\frac{1}{x_i(s)/\theta n-z_s(z_k)}-\frac{1}{x_i(s-)/\theta n-z_s(z_k)}\right)\\
&=\sum_{0< s\leq t\atop \Delta x_i(s)>0}\frac{1}{(\theta n)^2}\frac{1}{(x_i(s)/\theta n-z_s(z_j))^2(x_i(s)/\theta n-z_s(z_k))^2}+\OO\left(\frac{N^n_t}{  n^3}\right),
\end{split}\end{align}
where the implicit constant depends on $\theta$ and $\min_{1\leq j\leq m}|\Im[z_T(z_j)]|$.
We can further rewrite \eqref{e:quadvar2} as a sum of differences. For \eqref{e:var}, we have
\begin{align}\begin{split}\label{e:quadvarexp1}
&\phantom{{}={}}n^2[ M^n(z_j), M^n(z_j) ]_t
=\sum_{0< s\leq t\atop \Delta x_i(s)>0}\frac{1}{(\theta n)^2}\frac{1}{(x_i(s)/\theta n-z_s(z_j))^4}+\OO\left(\frac{N^n_t}{ n^3}\right)\\
&=\sum_{0< s\leq t\atop \Delta x_i(s)>0}-\frac{1}{3\theta n}\left(\frac{1}{(x_i(s)/\theta n-z_s(z_j))^3}-\frac{1}{(x_i(s-)/\theta n-z_s(z_j))^3}\right)+\OO\left(\frac{N^n_t}{ n^3}\right)\\
&=-\frac{1}{6\theta }\sum_{0< s\leq t\atop \Delta x_i(s)>0}\left(\del_z^2 m_{s}^n(z_s(z_j)) -\del_z^2 m_{s-}^n(z_s(z_j))\right)+\OO\left(\frac{N^n_t}{  n^3}\right)
\end{split}\end{align}
We recall the stopping times defined above Claim \ref{c:waittime}, and rewrite 
\eqref{e:quadvarexp1} as
\begin{align}\begin{split}\label{e:quadvarexp2}
&\phantom{{}={}}n^2[ M^n(z_j), M^n(z_j) ]_t
=-\frac{1}{6\theta }\sum_{0<l\leq N_t^n}\left(\del_z^2 m_{\tau^n_l}^n(z_{\tau^n_l}(z_j)) -\del_z^2 m_{\tau^n_{l-1}}^n(z_{\tau^n_l}(z_j))\right)+\OO\left(\frac{N^n_t}{  n^3}\right)\\
&=-\frac{1}{6\theta }\sum_{0<l\leq N_t^n}\left(\del_z^2 m_{\tau^n_l}^n(z_{\tau^n_l}(z_j)) -\del_z^2 m_{\tau^n_{l-1}}^n(z_{\tau^n_{l-1}}(z_j))-\del^3_zm_{\tau^n_{l-1}}^n(z_{\tau^n_{l-1}}(z_j))(z_{\tau^n_l}(z_j)-z_{\tau^n_{l-1}}(z_j))\right)\\
&+\OO\left(N_t^n\left(\frac{1}{ n^3}+\sup_{0<l\leq N_t^n} |z_{\tau^n_l}-z_{\tau^n_{l-1}}|^2\right)\right)\\
&=-\frac{1}{6\theta}\left(\del_z^2 m_{t}^n(z_{t}(z_j))-\del_z^2 m_{0}^n(z_{0}(z_j))-\int_0^t\del_z^3m_s^n(z_s(z_j))\rd z_s(z_j)\right)\\
&+\OO\left(N_t^n\left(\frac{1}{  n^3}+\sup_{0<l\leq N_t^n} |z_{\tau^n_l}-z_{\tau^n_{l-1}}|^2\right)+|z_t(z_j)-z_{\tau^n_{N_t^n}}(z_j)|\right)\\
&\Rightarrow -\frac{1}{6\theta }\int_0^t (\del_s\del^2_z m_s)(z_s(z_j))\rd s,
\end{split}\end{align}
where in the last line we used Claim \eqref{c:waittime} and that $z_t$ is Lipschitz with respect to $t$. This finishes the proof of  \eqref{e:var}.
For \eqref{e:cov}, we have
\begin{align*}
&\phantom{{}={}}n^2[ M^n(z_j), M^n(z_k) ]_t
=\sum_{0< s\leq t\atop \Delta x_i(s)>0}\frac{1}{(\theta n)^2}\frac{1}{(x_i(s)/\theta n-z_s(z_j))^2(x_i(s)/\theta n -z_s(z_k))^2}+\OO\left(\frac{N^n_t}{ n^3}\right)\\
&=\frac{1}{(\theta n)^2}\sum_{0< s\leq t\atop \Delta x_i(s)>0}\left(\frac{1}{(z_s(z_j)-z_s(z_k))^2}\left(\frac{1}{(x_i(s)/\theta n-z_s(z_j))^2}+\frac{1}{(x_i(s)/\theta n-z_s(z_k))^2}\right)\right.\\
&\left.-\frac{2}{(z_s(z_j)-z_s(z_k))^3}\left(\frac{1}{(x_i(s)/\theta n-z_s(z_j))}-\frac{1}{(x_i(s)/\theta n-z_s(z_k))}\right)\right)+\OO\left(\frac{N^n_t}{ n^3}\right)\\
&=-\frac{1}{\theta}\sum_{0<s\leq t\atop \Delta x_i(s)>0}
\frac{(m^n_s(z_s(z_j))-m^n_{s-}(z_s(z_j)))+(m^n_s(z_s(z_k))-m^n_{s-}(z_s(z_k)))}{(z_s(z_j)-z_s(z_k))^2}\\
&-
\frac{2}{\theta}\sum_{0<s\leq t\atop \Delta x_i(s)>0}\frac{(h^n_s(z_s(z_j))-h^n_{s-}(z_s(z_j)))-(h^n_s(z_s(z_k))-h^n_{s-}(z_s(z_k)))}{ (z_s(z_j)-z_s(z_k))^3}+\OO\left(\frac{N^n_t}{ n^3}\right),
\end{align*}
where $h_t^n(z)$ is the logarithmic potential of the empirical measure $\mu_t^n$,
\begin{align*}
h_t^n(z)=\int \ln(x-z)\rd \mu_t^n(x)=\frac{1}{n}\sum_{i=1}^n \ln (x_i(t)-z),\quad z\in \bC\setminus\bR.
\end{align*}
Thanks to Theorem \eqref{t:LLN}, we have 
\begin{align*}
h_t^n(z)\Rightarrow h_t(z)=\int \ln (x-z)\rd \mu_t,\quad z\in \bC\setminus\bR,
\end{align*}
where the logarithmic potential $h_t(z)$ is defined in \eqref{e:defht}.
By the same argument as in \eqref{e:quadvarexp2}, we get
\begin{align*}
n^2\langle M^n(z_j), M^n(z_k) \rangle_t
&\Rightarrow
-\frac{1}{\theta}\int_0^t \frac{(\del_s m_s)(z_s(z_j))+(\del_sm_s)(z_s(z_k))}{(z_s(z_j)-z_s(z_k))^2}+\frac{2((\del_s h_s)(z_s(z_j))-(\del_sh_s)(z_s(z_k)))}{(z_s(z_j)-z_s(z_k))^3}\rd s\\
&=-\frac{1}{\theta}\int_0^t\frac{(\del_s m_s)(z_s(z_j))+(\del_sm_s)(z_s(z_k))}{(z_s(z_j)-z_s(z_k))^2}-\frac{2(e^{-m_0(z_j)}-e^{-m_0(z_k)})}{(z_s(z_j)-z_s(z_k))^3}\rd s.
\end{align*}
This finishes the proof of \eqref{e:cov}.

\emph{Step three: subsequential limit. }

In the first step, we have proven that as $n$ goes to infinity, the random processes $\{(g^n_t(z_t(z_j)))_{1\leq j\leq m})\}_{0\leq t\leq T}$ are tight. Without loss of generality, by passing to a subsequence, we assume that they weakly converge towards to a random process $\{(\cG_j(t))_{1\leq j\leq m})\}_{0\leq t\leq T}$. We check that the limit process satisfies the stochastic differential equation \eqref{e:limitprocess}. The random process $\{(g^n_t(z_t(z_j)))_{1\leq j\leq m})\}_{0\leq t\leq T}$ satisfies the stochastic differential equation \eqref{e:gsde1}. For the first term on the righthand side of \eqref{e:gsde1}, by our assumption $(g_0^n(z_j))_{1\leq j\leq m}\Rightarrow (g_0(z_j))_{1\leq j\leq m}$. For the second term, by Theorem \ref{t:LLN}, we have
\begin{align*}
&\phantom{{}={}}\int_0^t \del_z m^n_s(z_s(z_j))e^{-m_0(z_j)} n\left(1-e^{-(m^n_s(z_s(z_j))-m_0(z_j))}\right)\rd s\\
&=\int_0^t \del_z m^n_s(z_s(z_j))e^{-m_0(z_j)} g^n_t(z_s(z_j))\rd s
+\OO\left(\int_0^tn|m_s^n(z_s(z_j))-m_0(z_j)|^2\rd s\right)\\
&\Rightarrow \int_0^t \del_z m_s(z_s(z_j))e^{-m_0(z_j)} \cG_j(s)\rd s
= \int_0^t \frac{\del_zm_0(z_j)e^{-m_0(z_j)}}{1-s\del_zm_0(z_j)e^{-m_0(z_j)}} \cG_j(s)\rd s
\end{align*}
where the last term vanishes, because the processes $\{(g^n_t(z_t(z_j)))_{1\leq j\leq m})\}_{0\leq t\leq T}$ are stochastically bounded, i.e. \eqref{e:stBound}.
For the third term, by Theorem \ref{t:LLN} and Proposition \ref{p:estimatemt}, we have
\begin{align*}\begin{split}
\int_{0}^t\left((\del_z m^n_s(z_s(z_j)))^2-\del_z^2 m^n_s(z_s(z_j))\right)e^{-m_s^n(z_s(z_j))}\rd s&\Rightarrow 
\int_{0}^t\left((\del_z m_s(z_s(z_j)))^2-\del_z^2 m_s(z_s(z_j))\right)e^{-m_0(z_j)}\rd s\\
&=\int_0^t\frac{((\del_z m_0(z_j))^2-\del_z^2 m_0(z_j))e^{-m_0(z_j)}}{(1-s\del_z m_0(z_j)e^{-m_0(z_j)})^3}\rd s.
\end{split}
\end{align*}
For the fourth term, in Step two we have proven that $\{(nM_t^n(z_j))_{1\leq j\leq m}\}_{0\leq t\leq T}$ converges weakly towards a centered complex Gaussian process $\{(\cW_j(t))_{1\leq j\leq m}\}_{0\leq t\leq T}$, which is characterized by \eqref{e:var} and \eqref{e:cov}. 
This finishes the proof of Theorem \ref{p:CLT}.
\end{proof}

\subsection{Extreme particles}\label{s:extremep}

In the following we first derive a large deviation estimate of the extreme particles of the $\beta$-nonintersecting random walks. Then Theorem \ref{t:CLT2} follows from Theorem \ref{t:CLT} by a contour integral.

\begin{proposition}\label{p:extremePbound}
Suppose the initial data $\bmx(0)$ satisfies Assumption \ref{a:ibound}. For any time $t>0$, there exists a constant $\fc$ depending on $\fb$ and $t$, such that
\begin{align}\label{e:extremePbound}
\fc n\geq x_1(t)\geq x_2(t)\geq \cdots \geq x_n(t), 
\end{align} 
with probability at least $1-\exp(-cn)$.
\end{proposition}

We notice that the $\beta$-nonintersecting Poisson random walks are shift invariant. Suppose that the $\beta$-nonintersecting Poisson random walk $\bmy(t)$ starts from $(a+(n-1)\theta, a+(n-2)\theta, \cdots, a)$, where $a= a(n)\in \bZ_{\geq 0}$. Then it follows from Theorem \ref{t:density} that for any fixed $t>0$, the law of $\bmy(t)$ is given by
\begin{align}\label{e:defPt}
\bP_{t}(y_1,y_2,\cdots, y_n)=\frac{1}{Z_n}\prod_{1\leq i <j\leq n}\frac{\Gamma(y_i-y_j+1)\Gamma(y_i-y_j+\theta)}{\Gamma(y_i-y_j)\Gamma(y_i-y_j+1-\theta)}\prod_{i=1}^n\frac{(\theta  t n)^{y_i-a}}{\Gamma(y_i-a+1)},
\end{align}
where the partition function $Z_n$ is given by
\begin{align*}
Z_n=e^{\theta  t n^2}(\theta  t n)^{\theta(n-1)n/2}\prod_{i=1}^n\frac{\Gamma(i\theta)}{\Gamma(\theta)}.
\end{align*}

The measure $\bP_t(y_1,y_2,\cdots, y_n)$ is a discrete $\beta$ ensemble studies in \cite{MR3668648}. The next proposition follows from \cite[Theorem 7.1]{MR3668648}.
\begin{proposition}\label{p:0initialbound}
Take $a=\lceil \fb n\rceil$ and $ t>0$. There exits a constant $\fc$ depending $\fb$ and $t$, such that the measure $\bP_ t$ as in \eqref{e:defPt} satisfies  
\begin{align*}
\bP_{ t}\left(y_1\leq \fc n\right)\geq 1-\exp(-cn). 
\end{align*} 
\end{proposition}

\begin{proof}[Proof of Proposition \ref{p:extremePbound}]
Let $\bmx(t)$ be a $\beta$-nonintersecting Poisson random walk with initial data $\bmx(0)\in \bW_\theta^n$ satisfying \eqref{e:ibound}, and  $\bmy(t)$ another independent $\beta$-nonintersecting Poisson random walk with initial data $\bmy(0)=(\lceil \fb n\rceil+(n-1)\theta, \lceil\fb n\rceil+(n-2)\theta, \cdots, \lceil \fb n \rceil)$. Let $\fc$ be as in Proposition \ref{p:0initialbound}, we prove by constructing a coupling of $\bmx(t)$ and $\bmy(t)$, that 
\begin{align}\label{e:PBound}
\bP(x_1( t)\leq \fc n)\geq \bP(y_1( t)\leq \fc n).
\end{align}
Then the claim \eqref{e:extremePbound} follows from combining Proposition \ref{p:0initialbound} and \eqref{e:PBound}.


We define the coupling $(\hat\bmx(t), \hat\bmy(t))$ as a Poisson random walk on $\bW^n_\theta\times \bW^n_\theta$, with initial data $(\hat\bmx(0), \hat\bmy(0))=(\bmx(0), \bmx(0))$, and generator
\begin{align*}\begin{split}
\hat\cL^n_\theta f(\bmx, \bmy)&=\theta n\sum_{i=1}^n\left[\frac{V(\bmx+\theta \bme_i)}{V(\bmx)}-\frac{V(\bmy+\theta \bme_i)}{V(\bmy)}\right]_+\left(f(\bmx+\bme_i, \bmy)-f(\bmx, \bmy)\right)\\
&+\theta n\sum_{i=1}^n\left[\frac{V(\bmy+\theta \bme_i)}{V(\bmy)}-\frac{V(\bmx+\theta \bme_i)}{V(\bmx)}\right]_+\left(f(\bmx, \bmy+\bme_i)-f(\bmx, \bmy)\right)\\
&+\theta n\sum_{i=1}^n\min\left\{\frac{V(\bmx+\theta \bme_i)}{V(\bmx)},\frac{V(\bmy+\theta \bme_i)}{V(\bmy)}\right\}\left(f(\bmx+\bme_i, \bmy+\bme_i)-f(\bmx, \bmy)\right).
\end{split}\end{align*}
where $[x]_+=\max\{x,0\}$. The marginal distributions of $\hat\bmx(t)$ and $\hat\bmy(t)$ coincide with those of $\bmx(t)$ and $\bmy(t)$ respectively, 
\begin{align*}
\{\hat\bmx(s)\}_{0\leq s\leq t}\overset{d}{=}\{\bmx(s)\}_{0\leq s\leq t},\quad \{\hat\bmy(s)\}_{0\leq s\leq t}\overset{d}{=}\{\bmy(s)\}_{0\leq s\leq t}.
\end{align*}

For the initial data, we have 
\begin{align*}
\hat x_i(0)\leq \fb n\leq \hat y_i(0), \quad 1\leq i\leq n.
\end{align*}
In the following we prove that the coupling process $(\hat \bmx(t), \hat\bmy(t))$ satisfies
\begin{align}\label{e:comparison1}
\bP\left(\text{for all $t\geq 0$ and $1\leq i\leq n$, }\hat x_i(t)\leq \hat y_i(t)\right)=1.
\end{align}
We define a sequence of stopping times, $\tau^n_1, \tau^n_2, \tau^n_3, \cdots$, where $\tau^n_k$  is the time of the $k$-th jump of the coupling process $(\hat \bmx(t), \hat \bmy(t))$. We prove by induction that 
\begin{align}\label{e:comparison2}
\bP\left(\text{for all $0\leq t\leq \tau^n_k$ and $1\leq i\leq n$, }\hat x_i(t)\leq \hat y_i(t)\right)=1.
\end{align}
Then \eqref{e:comparison1} follows by noticing that $\lim_{k\rightarrow\infty}\tau^n_k=\infty$. We assume that \eqref{e:comparison2} holds for $k$, we prove it for $k+1$. If $\hat x_i(\tau^n_k)<\hat y_i(\tau^n_k)$, then with probability one, $\hat x_i(\tau^n_{k+1})\leq \hat y_i(\tau^n_{k+1})$. If $\hat x_i(\tau^n_k)=\hat y_i(\tau^n_k)$, by our assumptions, $\hat\bmx(\tau^n_k), \hat \bmy(\tau^n_k)\in \bW^n_\theta$ and $\hat x_j(\tau^n_k)\leq \hat y_j(\tau^n_k)$ for all $1\leq j\leq n$, we have
\begin{align*}
0\leq \frac{\hat x_i(\tau^n_k)-\hat x_j(\tau^n_k)+\theta}{\hat x_i(\tau^n_k)-\hat x_j(\tau^n_k)}\leq \frac{\hat y_i(\tau^n_k)-\hat y_j(\tau^n_k)+\theta}{\hat y_i(\tau^n_k)-\hat y_j(\tau^n_k)},\quad j\neq i.
\end{align*}
Thus the jump rate from $(\hat \bmx(\tau^n_k), \hat\bmy(\tau^n_k))$ to $(\hat \bmx(\tau^n_k)+\bme_i, \hat\bmy(\tau^n_k))$,
\begin{align*}\begin{split}
&\phantom{{}={}}\theta n\left[\frac{V(\hat \bmx(\tau^n_k)+\theta \bme_i)}{V(\hat \bmx(\tau^n_k))}-\frac{V(\hat\bmy(\tau^n_k)+\theta \bme_i)}{V(\hat \bmy(\tau^n_k))}\right]_+\\
&=\theta n\left[\prod_{j:j\neq i}\frac{\hat x_i(\tau^n_k)-\hat x_j(\tau^n_k)+\theta}{\hat x_i(\tau^n_k)-\hat x_j(\tau^n_k)}-\prod_{j:j\neq i}\frac{\hat y_i(\tau^n_k)-\hat y_j(\tau^n_k)+\theta}{\hat y_i(\tau^n_k)-\hat y_j(\tau^n_k)}\right]_+= 0.
\end{split}\end{align*}
vanishes. Therefore with probability one, $\hat x_i(\tau^n_{k+1})\leq \hat y_i(\tau^n_{k+1})$.
This finishes the proof of \eqref{e:comparison1} and \eqref{e:comparison2}.

It follows from \eqref{e:comparison1}, 
\begin{align}\label{e:PBound2}
\bP(\hat x_1( t )\leq \fc n)\geq \bP(\hat y_1( t )\leq \fc n).
\end{align}
Since the marginal distributions of $\hat\bmx(t)$ and $\hat\bmy(t)$ coincide with those of $\bmx(t)$ and $\bmy(t)$ respectively, \eqref{e:PBound} follows from combining Proposition \ref{p:0initialbound} and \eqref{e:PBound2}. This finishes the proof of Propostion \ref{p:extremePbound}.
\end{proof}

\begin{proof}[Proof of Theorem \ref{t:CLT2}]
We take a contour $\cC$ which encloses a neighborhood of $[0,\fc/\theta]$. Then with exponentially high probability we have
\begin{align*}
n\int f_j(x)\rd(\mu_t^n(x)-\mu_t(x))=\frac{1}{2\pi\ri}\oint_\cC g^n_t(w)f_j(w)\rd w, \quad 1\leq j\leq m.
\end{align*}
By Proposition \ref{e:defOmega}, $z_t(z)$ is a homeomorphism from the closure of $\Omega_t\cap \bC_+$ to $\bC_+\cup \bR$, and from the closure of $\Omega_t\cap \bC_-$ to $\bC_-\cup \bR$.  By a change of variable, we have
\begin{align*}
\frac{1}{2\pi\ri}\oint_\cC g^n_t(w)f_j(w)\rd w=\frac{1}{2\pi\ri}\oint_{z_t^{-1}(\cC)} g^n_t(z_t(z))f_j(z_t(z))\rd z_t(z), \quad 1\leq j\leq m.
\end{align*} 
By the continuous mapping theorem of weak convergence, it follows from Theorem \ref{t:CLT}
\begin{align*}
&\left\{\left(\frac{1}{2\pi\ri}\oint_{z_t^{-1}(\cC)} g^n_t(z_t(z))f_j(z_t(z))\rd z_t(z)\right)_{1\leq j\leq m}\right\}_{0\leq t\leq T}\Rightarrow \{(\cF_j(t))_{1\leq j\leq m}\}_{0\leq t\leq T}\\
&\cF_j(t)\deq \frac{1}{2\pi\ri}\oint_{z_t^{-1}(\cC)} g_t(z_t(z))f_j(z_t(z))\rd z_t(z),\quad 1\leq j\leq m,
\end{align*}
and the means and the covariances of the Gaussian process $\{(\cF_j(t))_{1\leq j\leq m}\}_{0\leq t\leq T}$ are given by
\begin{align*}
\bE[\cF_j(t)]&=\frac{1}{2\pi\ri}\oint_{z_t^{-1}(\cC)}\mu(t, z) f_j(z_t(z))\rd z_t(z)=\frac{1}{2\pi\ri}\oint_{\cC}\mu(t, z_t^{-1}(w))f_j(w)\rd w\\
\cov[\cF_j(s), \cF_{k}(t)]&=-\frac{1}{4\pi^2}\oint_{z_s^{-1}(\cC)}\oint_{z_t^{-1}(\cC)}\sigma(s, z,t,z') f_j(z_s(z))f_k(z_t(z'))\rd z_s(z)\rd z_t(z')\\
&=-\frac{1}{4\pi^2}\oint_{\cC}\oint_{\cC}\sigma(s, z_s^{-1}(w),t,z_t^{-1}(w')) f_j(w)f_k(w')\rd w\rd w'.
\end{align*}
where $\mu(t,z)$ and $\sigma(s, z,t,z')$ are as defined in \eqref{e:mean0} and \eqref{e:variance0}. This finishes the proof of Theorem \ref{t:CLT2}.
\end{proof}

\begin{proof}[Proof of Corollary \ref{c:GFF}]
We denote $\tilde f_j(x)=\int_0^x f_j(x)\rd x$ an anti-derivative of $f_j$ for $1\leq j\leq m$. We perform an integration by part,
\begin{align*}
\left(\sqrt{\pi\theta}\int_\bR f_j(x) (H_n(x,t)-\bE[H_n(x,t)])\rd x\right)_{1\leq j\leq m}
=\left(n\sqrt{\pi\theta}\int_{\bR} \tilde f_j(x)\left(\rd \mu^n_t(x)-\bE[\rd \mu^n_t(x)]\right)\right)_{1\leq j\leq m},
\end{align*}
which converges as $n$ goes to infinity in $D([0,T],\bR^m)$ weakly towards a Gaussian process, by Theorem \ref{t:CLT2}. In the following, we identify its covariance structure. By Theorem \ref{t:CLT2}, for any $s< t$ (the case $s=t$ follows by taking limit), the covariance 
\begin{align}\begin{split}\label{e:covcc}
&\phantom{{}={}}\cov\left[n\sqrt{\pi\theta}\int_{\bR} \tilde f_j(x)\left(\rd \mu^n_s(x)-\bE[\rd \mu^n_s(x)]\right), n\sqrt{\pi\theta}\int_{\bR} \tilde f_k(x)\left(\rd \mu^n_t(x)-\bE[\rd \mu^n_t(x)]\right)\right]\\
&=-\frac{1}{4\pi}\oint_{z_t^{-1}(\cC)}\oint_{z_s^{-1}(\cC)}\del_{z}\del_{z'}\log \left(\frac{z-z'}{z_{s}(z)-z_{s}(z')}\right)\tilde f_j(z_s(z))\tilde f_k(z_t(z'))\rd z\rd z'\\
&=-\frac{1}{4\pi}\oint_{z_t^{-1}(\cC)}\oint_{z_s^{-1}(\cC)}\del_{z}\del_{z'}\log \left(z-z'\right)\tilde f_j(z_s(z))\tilde f_k(z_t(z'))\rd z\rd z',
\end{split}\end{align}
where we used $z_s(z_t^{-1}(\cC))$ is outside the contour $\cC$, thus
\begin{align*}
&\phantom{{}={}}\oint_{z_t^{-1}(\cC)}\oint_{z_s^{-1}(\cC)}\del_{z}\del_{z'}\log \left(z_s(z)-z_s(z')\right)\tilde f_j(z_s(z))\tilde f_k(z_t(z'))\rd z\rd z'\\
&=\oint_{z_t^{-1}(\cC)}\left(\oint_{\cC}\frac{\tilde f_j(z)}{(z-z_s(z'))^2}\rd z\right) \tilde f_k(z_t(z'))\rd z_s(z')=0.
\end{align*}
We recall that $z_t(z)=z+te^{-m_0(z)}$ is conformal from $\Omega_t$ as defined in \eqref{e:defOmega} to $\bC\setminus \bR$, and is a homeomorphism from the closure of $\Omega_t\cap \bC_+$ to $\bC_+\cup \bR$, and from the closure of $\Omega_t\cup \bC_-$ to $\bC_-\cup \bR$.
We denote $\gamma_t=\del \Omega_t \cap \bC_+$, the boundary of $\Omega_t$ in the upper half plane. From \eqref{e:defOmega}, $\gamma_t$ is explicitly given by
\begin{align*}
\gamma_t \deq\left\{z\in \bC_+: \int \frac{\rd Q(\mu_0)(x)}{|x-z|^2}=\frac{1}{t}\right\}=\{z\in \bC_+: t=t(z)\},
\end{align*}
and $z_t(\gamma_t)\in \bR$.
We can deform the contours in \eqref{e:covcc}, and perform an integration by part
\begin{align}\begin{split}\label{e:covcc2}
&\phantom{{}={}}\cov\left[n\sqrt{\pi\theta}\int_{\bR} \tilde f_j(x)\left(\rd \mu^n_s(x)-\bE[\rd \mu^n_s(x)]\right), n\sqrt{\pi\theta}\int_{\bR} \tilde f_k(x)\left(\rd \mu^n_t(x)-\bE[\rd \mu^n_t(x)]\right)\right]\\
&=-\frac{1}{4\pi}\oint_{\gamma_t\cup\bar \gamma_t}\oint_{\gamma_s\cup\bar \gamma_s}\del_{z}\del_{z'}\log \left(z-z'\right)\tilde f_j(z_s(z))\tilde f_k(z_t(z'))\rd z\rd z'\\
&=-\frac{1}{4\pi}\oint_{\gamma_t\cup\bar \gamma_t}\oint_{\gamma_s\cup\bar \gamma_s}\log \left(z-z'\right) f_j(z_s(z))f_k(z_t(z'))\rd z_s(z)\rd z_t(z').
\end{split}\end{align}
We notice that on those contours $z\in \gamma_s\cup\bar \gamma_s$, or $z'\in \gamma_t\cup\bar \gamma_t$, $z_s(z)=x(z)$ and $z_t(z')=x(z')$ are real. Using this fact and the equality
\begin{align*}
2\log \left|\frac{z-z'}{z-\bar z'}\right|=\log \frac{(z-w)(\bar z-\bar z')}{(z-\bar z')(\bar z-z')},
\end{align*}
we can rewrite \eqref{e:covcc2}
\begin{align*}
&\phantom{{}={}}\cov\left[n\sqrt{\pi\theta}\int_{\bR} \tilde f_j(x)\left(\rd \mu^n_s(x)-\bE[\rd \mu^n_s(x)]\right), n\sqrt{\pi\theta}\int_{\bR} \tilde f_k(x)\left(\rd \mu^n_t(x)-\bE[\rd \mu^n_t(x)]\right)\right]\\
&=-\frac{1}{2\pi}\int_{\gamma_t}\int_{\gamma_s}\log \left|\frac{z-z'}{z-\bar z'}\right| f_j(z_s(z))f_k(z_t(z'))\rd z_s(z)\rd z_t(z')\\
&=\int_{z'\in \bH, t(z)=t}\int_{z\in\bH,t(z)=s}G(z,z') f_j(x(z))f_k(x(z'))\rd x(z)\rd x(z')\\
&=\cov\left[\int_{z\in \bH, t(z)=s} f_j(x(z))\fG(z) \rd x(z), \int_{z\in \bH, t(z)=t} f_k(x(z))\fG(z) \rd x(z)\right]
\end{align*}
This finishes the proof of Corollary \ref{c:GFF}.

\end{proof}

\bibliography{References}{}
\bibliographystyle{plain}

\end{document}